\documentclass[11pt,leqno,a4paper]{article}

\makeatletter
  
  \@addtoreset{equation}{section}
\makeatother

\usepackage{upref,amsxtra,amscd,amsopn,amsbsy,amstext,amssymb,amsmath,amsthm}
\usepackage{bm}

\newtheorem{thm}{Theorem}[section]
\newtheorem{lem}{Lemma}[section]
\newtheorem{prop}{Proposition}[section]

\newtheorem{rem}{Remark}[section]

\newtheorem{dfn}{Definition}[section]

\newcommand{\pd}{\partial}
\newcommand{\gm}{\gamma}

\newcommand{\lm}{\lambda}
\newcommand{\R}{\mathbb{R}}

\newcommand{\N}{\mathbb{N}}
\newcommand{\p}[3]{\dfrac{\partial^{#3} #1}{\partial {#2}^{#3}}}

\newcommand{\av}[1]{\left| #1 \right|}

\newcommand{\Ln}[2]{\left\| #1 \right\|_{L^2(#2)}}
\newcommand{\Li}[2]{\left\| #1 \right\|_{L^{\infty}(#2)}}
\newcommand{\Lp}[3]{\left\| #1 \right\|_{L^{#2}(#3)}}
\newcommand{\Hn}[2]{\left\| #1 \right\|_{H^{#2}(\Omega)}}
\newcommand{\Hnd}[1]{\left\| #1 \right\|_{H^{2}_0(\Omega)}}
\newcommand{\Cn}[3]{\left\| #1 \right\|_{C^{#2}(#3)}}
\newcommand{\vk}{\kappa}
\newcommand{\vp}{\varphi}
\newcommand{\vq}{\psi}
\newcommand{\va}{\alpha}
\newcommand{\vb}{\beta}
\newcommand{\vd}{\delta}
\newcommand{\vz}{\zeta}
\newcommand{\vs}{\theta}

\newcommand{\ve}{\varepsilon}

\newcommand{\mC}{\mathcal{C}}

\newcommand{\mM}{\mathcal{M}}
\newcommand{\mN}{\mathcal{N}}

\newcommand{\con}[1]{\overline{#1}}

\allowdisplaybreaks

\begin{document}

\title{\bf Regularity of the obstacle problem \\ for the parabolic biharmonic equation}
\author{{\sc M. Novaga} and {\sc S. Okabe}}
\date{}

\maketitle

\begin{abstract}
We study the regularity of solutions to the obstacle problem for the parabolic biharmonic equation. 
We analyze the problem via an implicit time discretization, and  
we prove some regularity properties of the solution. 
\end{abstract}



\section{Introduction}

The purpose of this paper is to investigate the regularity properties 
of solutions to the obstacle problem for the parabolic biharmonic equation.

The parabolic biharmonic equation is a prototype of higher order parabolic equations, 
and has been intensively studied in the mathematical literature. We refer 
for instance to \cite{Ba,CM,CR,G,GG,GGS,GH,GP,T} and references therein, for a nonexhaustive list of works on this 
equation, and for a discussion of possible applications.

The obstacle problem for elliptic and parabolic PDE's is a topics which attracted a great interest 
in the past years. 
However, even if many studies are available on second order elliptic and parabolic equations
(see for instance \cite{C,CPS} and references therein), 
there are relatively few results for higher order obstacle problems, even in the linear fourth order case. 
In particular, while the elliptic obstacle problem for the biharmonic operator has been considered in 
\cite{BS,CF,CFT,F1,Schi},
to the best of our knowledge no result is available for the corresponding parabolic obstacle problem. 

We let $\Omega \subset \R^N$ be a bounded domain, with boundary of class $C^2$, 
and we let $f : \Omega \to \R$ be the obstacle function, satisfying 
\begin{align} \label{f-cond}
f \in C^2(\overline{\Omega}), \qquad f < 0 \,\,\, \text{on} \,\,\, \pd \Omega. 
\end{align}
We consider an initial datum $u_0 : \Omega \to \R$ such that
\begin{align} \label{initial-cond}
u_0 \in H^2_0(\Omega),  \qquad u_0 \ge f \,\, \text{a.e. in} \,\,{\Omega}.    
\end{align}
We recall that $u \in H^2_0(\Omega)$ implies $u=0$ and $\nabla u\cdot\nu^\Omega=0$ (weakly) 
on $\pd \Omega$, that is, $u$ satisfies the so-called Dirichlet boundary conditions on $\pd\Omega$ (see \cite{ADN,GGS}), 
where $\nu^{\Omega}$ denotes the unit outer normal of $\pd \Omega$. 

We shall consider the following fourth order parabolic obstacle problem:  
\begin{align} \label{P}
\begin{cases}
u_t(x,t) +\Delta^2 u(x,t)\ge 0 \quad & \text{in} \quad \Omega \times \R_+, \\
u_t(x,t)+ \Delta^2 u(x,t) =0\quad & \text{in} \quad \{(x,t)\in\Omega \times \R_+:\ u(x,t)>f(x)\}, \\
u(x,t)=0 \quad & \text{on} \quad \pd \Omega \times \R_+, \\
\nabla u(x,t)\cdot \nu^\Omega(x)=0 \quad & \text{on} \quad \pd \Omega \times \R_+, \\
u(x,t) \ge f(x) \quad & \text{in} \quad \Omega \times \R_+, \\
u(x,0)=u_0(x) \quad & \text{in} \quad \Omega. 
\end{cases} \tag{P}
\end{align}
In order to state the main result of this paper precisely, we define a weak solution of \eqref{P}. 
Let us set 
\begin{align}
\mathcal{K}
&:= \{ u \in L^{2}(0,T;H^{2}_{0}(\Omega)) \mid u_{t} \in L^{2}(\Omega \times (0,T)), \,\, 
      u \ge f \,\, \text{a.e. in} \,\, \Omega \times (0,T), \\
& \qquad \qquad \qquad \qquad \qquad \qquad 
      u(x,0) = u_{0}(x) \,\, \text{a.e. in} \,\, \Omega \} \notag
\end{align}
Then a weak solution of \eqref{P} is defined as follows:  
\begin{dfn} \label{def-weak}
$u$ is a weak solution of \eqref{P} if 
\begin{enumerate}
\item[{\rm (i)}] $u \in \mathcal{K}$, 
\item[{\rm (ii)}] For any $w \in \mathcal{K}$, it holds that 
\begin{align} \label{vari-ineq}
\int^{T}_{0}\!\!\!\! \int_{\Omega} \left[ u_t (w-u) + \Delta u \Delta(w-u) \right] \, dx dt \ge 0.   
\end{align}
\end{enumerate}
\end{dfn}
We now state the main result of this paper. 
\begin{thm} \label{main-thm}
Let $N \ge 1$. 
Let $f$ be a function satisfying \eqref{f-cond}. 
Then, for any initial data $u_0$ satisfying \eqref{initial-cond}, the problem \eqref{P} has a 
unique weak solution 
\begin{align} \label{MMO}
u \in L^\infty(\R_+;H^2_0(\Omega))\cap H^1_{loc}(\R_+;L^2(\Omega)), \quad {with\ }\,
u_t \in L^2(\R_+\times\Omega).
\end{align}
Furthermore, for a.e. $t \in \R_+$ the quantity
\begin{equation}\label{def-measure}
\mu_t := u_t(\cdot,t) + \Delta^2 u(\cdot,t)
\end{equation}
defines a Radon measure in $\Omega$, and for any $T>0$ there exists a constant $C>0$ such that  
\begin{align} \label{bdd-measure-sense}
\int^T_0 \mu_t(\Omega)^2 dt < C.   
\end{align}
Moreover, when $N \le 3$, the following regularity properties hold: 
\begin{enumerate}
\item[{\rm (i)}] $u \in L^2(0,T;W^{2, \infty}(\Omega))$ for any $T<+\infty$. 
In particular, if $N=1$, 
\begin{align} \label{conti-u-N1}
u \in C^{0, \vb}([0,T]; C^{1, \gm}(\Omega) \,\,\, \text{with} \,\,\, 0< \gm < \frac{1}{2} 
 \quad \text{and} \quad 0 < \vb < \frac{1-2\gm}{8}, 
\end{align}
if $N\in\{2,3\}$, 
\begin{align} \label{conti-N23}
u \in C^{0, \vb}([0,T]; C^{0, \gm}(\Omega)) \,\,\, \text{with} \,\,\, 0< \gm < \frac{4-N}{2} 
 \quad \text{and} \quad 0 < \vb < \frac{4-N-2\gm}{8}.  
\end{align}
\item[{\rm (ii)}] For any $0<T<+\infty$, it holds that 
\begin{align}\label{mu-supp}
{\rm supp}\, \mu_t \subset \{ (x,t) \in \Omega \times (0,T) \mid u(x,t) = f(x) \}
\end{align}
and $u$ satisfies \eqref{P} in the sense of distribution. 
\end{enumerate}
\end{thm}

We need to impose the restriction on the dimension $N\le 3$
in order to obtain the $W^{2,\infty}$ estimate on the solution $u(\cdot,t)$
(see Remark \ref{rem} for further comments on this). 
However, in analogy with the regularity results in the stationary case
\cite{F1,CF}, one may expect that the $W^{2,\infty}$ estimate holds in any dimension. 
\smallskip

Let us point out that problem \eqref{P} corresponds to the gradient flow of a convex functional 
defined on the Hilbert space $L^2(\Omega)$, hence we can apply the general theory of maximal monotone 
operators developed in \cite{B}. Indeed, given $f$ as above, 
we can define the functional $E_f(u):L^{2}(\Omega) \to [0,+\infty]$ as 
\begin{align*}
E_f(u)=
\begin{cases}
\dfrac{1}{2} \displaystyle\int_\Omega |\Delta u|^2 \quad & \text{if} \quad 
     u\in H^2_0(\Omega) \quad \text{and} \quad u\geq f, \\
+\infty \quad & \text{otherwise}.
\end{cases}
\end{align*}
Notice that $E_f(u)$ is convex and lower semicontinuous on $L^{2}(\Omega)$, 
and the problem \eqref{P} corresponds to the gradient flow 
\begin{equation}\label{MF}
u_t + \pd E_f(u) \ni 0\,, \qquad\qquad u(0)= u_0\,,
\end{equation}
where $\pd E_f$ denotes the subdifferential of $E_f$ in $L^2(\Omega)$. 
In particular, given an initial datum $u_0\in H^2_0(\Omega)$ with $u_0 \ge f$, by the results in \cite{B} 
it follows that the evolution problem \eqref{MF} has a unique solution $u$ satisfying \eqref{MMO}. 

In this paper we characterize the solution $u$ by means of an implicit variational scheme, 
corresponding to the minimizing movements introduced by De Giorgi (see e.g. \cite{A}). 
This approach will allow us to extend some of the arguments in \cite{CF}, 
concerning the regularity of the elliptic obstacle problem for the biharmonic operator. 
We point out that the method does not rely on the linear structure of the problem and can be applied 
to more general fourth order parabolic equations. 
Indeed, one motivation for this work comes from the motion of planar closed curves by
the elastic flow, in presence of obstacles.  The elastic flow is the $L^{2}$ gradient flow of the elastic energy
\begin{align*}
\mathcal E(\gm)\ =\  \int_\gm \vk^2 \, ds, 
\end{align*}
where $\gm$ is a  planar closed curve and $\vk$ denotes the curvature of $\gm$. 
Among other applications, this flow models the evolution of lipid bilayer membranes
(see for instance \cite{ES}), where the presence of obstacles is a natural features.

Although this flow is governed by a fourth order quasilinear parabolic equation, 
we expect that the method of this paper can be adapted, and this will be subject 
of future investigation. 

The paper is organized as follows: in Section 2 we introduce the implicit scheme corresponding to 
problem \eqref{P}, by means of an appropriate variational problem; 
in Section 3 we study the regularity of solutions to the variational problem; 
in Section 4 we pass to the limit in the approximating scheme and prove Theorem \ref{main-thm}.





\subsection{Notation}
The equation in \eqref{P} is the $L^2$ gradient flow for the functional 
\begin{align*}
E(u)= \dfrac{1}{2} \int_\Omega \av{\Delta u(x)}^2 \, dx. 
\end{align*}
Let $T>0$, $n \in \N$, and set 
\begin{align*}
\tau_n = \dfrac{T}{n}. 
\end{align*}
Let us set $u_{0,n} = u_0$. For $i= 1, \cdots, n$, we define inductively $u_{i, n}$ as a solution of the minimum problem 
\begin{align} \label{incre-mini}
\min{\{ G_{i, n}(u) :\  u \in K \}} \tag{$M_{i,n}$},
\end{align}
where 
\begin{align} \label{G-func}
G_{i,n}(u):= E(u) + P_{i,n}(u) 
\end{align}
with 
\begin{align} \label{to-v}
P_{i,n}(u) := \dfrac{1}{2 \tau_n} \int_\Omega (u - u_{i-1,n})^2 \, dx,  
\end{align}
and $K$ is a convex set given by 
\begin{align*}
K:= \{ u \in H^2_0(\Omega) :\   u(x) \ge f(x) \,\, \text{a.e. in} \,\, \Omega\}. 
\end{align*}
In the following, we let
\begin{align} \label{discrete-velocity}
V_{i, n}(x):= \dfrac{u_{i, n}(x) - u_{i-1, n}(x)}{\tau_n}. 
\end{align}
\begin{dfn}{\bf (Piecewise linear interpolation)} \label{def-linear-inter}
Define $u_n : \Omega \times [0,T] \to \R$ as 
\begin{align} \label{discrete}
u_n(x,t):= u_{i-1,n}(x) + (t-(i-1)\tau_n) V_{i,n}(x)  
\end{align}
if $(x,t) \in \Omega \times [(i-1) \tau_n, i \tau_n]$ for $i= 1, \cdots, n$. 
\end{dfn}
\begin{dfn}{\bf (Piecewise constant interpolation)} \label{def-piece-inter}
Define $\tilde{u}_n : \Omega \times [0,T] \to \R$ as 
\begin{align} 
\tilde{u}_n(x,t)&:= u_{i,n}(x), \label{discrete-2} \\ 
V_n(x,t) & := V_{i, n}(x), \label{evolv-velocity}
\end{align}
if $(x,t) \in \Omega \times [(i-1) \tau_n, i \tau_n)$ for $i= 1, \cdots, n$. 
\end{dfn}


\section{Existence and regularity of minimizers of \eqref{incre-mini}}
We first mention a well-known compactness result in $H^2_0(\Omega)$ \cite{AF,ADN}. 
\begin{prop} \label{compactness}
The following embedding is compact\,{\rm :} 
\begin{align} \label{H20-compact}
H^2_0(\Omega) \hookrightarrow 
\begin{cases}
 C^{1,\gm}(\overline{\Omega}) \quad \text{for} \quad 0 < \gm < \dfrac{1}{2} & \quad \text{if} \quad N=1, \\
 C^{0,\gm}(\Omega) \quad \text{for} \quad 0 < \gm < 2 - \dfrac{N}{2} & \quad \text{if} \quad N=2, 3, \\
 L^q(\Omega) \quad \,\,\,\, \text{for} \quad 1 \le \forall q < +\infty & \quad \text{if} \quad N=4, \\  
 L^q(\Omega) \quad \,\,\,\, \text{for} \quad 1 \le \forall q < \dfrac{2N}{N-4} & \quad \text{if} \quad N \ge 5. 
\end{cases}
\end{align} 
\end{prop}

We now show the existence of minimizers of \eqref{incre-mini}. 
\begin{thm}{\bf (Existence of minimizers)} \label{existence-mini}
Let $f$ be a function satisfying \eqref{f-cond}. 
Let $u_0$ satisfy \eqref{initial-cond}. 
Then the problem \eqref{incre-mini} possesses a unique solution 
$u_{i,n} \in H^2_0(\Omega)$ with $u_{i,n}(x) \ge f(x)$ a.e. in $\Omega$ for each $i= 1, \cdots, n$. 
\end{thm}
\begin{proof}
Fix $n \in \N$, $T>0$, and $i=1, \cdots, n$, arbitrarily. 
From \eqref{G-func}-\eqref{to-v} and the minimality of a solution $u$ to \eqref{incre-mini}, we obtain that 
\begin{align*}
E(u) \le G_{i, n}(u) \le G_{i,n}(u_{i-1,n}) = E(u_{i-1, n}),  
\end{align*}
and then 
\begin{align*}
0 \le \inf_{H^2_0(\Omega)}{G_{i, n}(u)} \le G_{i,n}(u_{i-1,n})= E(u_{i-1, n}) \le \cdots \le E(u_0). 
\end{align*}
Thus we can take a minimizing sequence $\{ u_j \} \subset H^2_0(\Omega)$ for \eqref{incre-mini} such that 
$u_j(x) \ge f(x)$ a.e. in $\Omega$ for each $j \in \N$ and $\sup_{j}{G_{i,n}(u_j)} < \infty$. 

Observing that the norm $\Ln{\Delta u}{\Omega}$ is equivalent to $\| u\|_{H^2_0(\Omega)}$ (see \cite{Li}), 
it follows from 
\begin{align*}
\Ln{\Delta u_j}{\Omega} = \sqrt{2 E(u_j)} \le \sqrt{2 E(u_0)} =  \Ln{\Delta u_0}{\Omega} 
\end{align*}
that $\{ u_j \}$ is uniformly bounded in $H^2_0(\Omega)$. 
Thus there exists $u \in H^2_0(\Omega)$ such that 
\begin{align} \label{u-H20-weak}
& u_j \rightharpoonup u \quad \text{in} \quad H^2_0(\Omega), 
\end{align}
in particular, 
\begin{align}\label{u-D2-weak}
\Delta u_j \rightharpoonup \Delta u \quad \text{in} \quad L^2(\Omega), 
\end{align}
up to a subsequence. Thanks to Proposition \ref{compactness}, we obtain that 
\begin{align*}
u_j \to u \quad \text{in} \quad 
\begin{cases}
 C^{1,\gm}(\bar{\Omega}) \quad \text{for} \quad 0 < \gm < \dfrac{1}{2} & \quad \text{if} \quad N=1, \\
 C^{0,\gm}(\Omega) \quad \text{for} \quad 0 < \gm < 2 - \dfrac{N}{2} & \quad \text{if} \quad N=2, 3, \\
 L^q(\Omega) \quad \,\,\,\, \text{for} \quad 1 \le \forall q < +\infty & \quad \text{if} \quad N=4, \\  
 L^q(\Omega) \quad \,\,\,\, \text{for} \quad 1 \le \forall q < \dfrac{2N}{N-4} & \quad \text{if} \quad N \ge 5. 
\end{cases}
\end{align*}
In particular 
\begin{align} \label{u-m-c}
u_j \to u \quad \text{a.e. in} \,\,\, \Omega \,\,\, \text{up to a subsequence.} 
\end{align}
Recalling $u_j \ge f$ a.e. in $\Omega$ for each $j \in \N$, \eqref{u-m-c} yields that $u \ge f$ a.e. in $\Omega$. 
Making use of Fatou's Lemma, we conclude that 
\begin{align} \label{mini-1}
P_{i, n}(u) \le \liminf_{j \to \infty}{P_{i, n}(u_j)}.  
\end{align}
Furthermore \eqref{u-D2-weak} implies 
\begin{align} \label{mini-2}
E(u) = \dfrac{1}{2} \Ln{\Delta u}{\Omega}^2 
\le \dfrac{1}{2} \liminf_{j \to \infty}{ \Ln{\Delta u_j}{\Omega}^2}
= \liminf_{j \to \infty}{E(u_j)}. 
\end{align}
Combining \eqref{mini-1} with \eqref{mini-2}, 
we see that $u \in H^2_0(\Omega)$ is the minimizer of \eqref{incre-mini} with $u \ge f$ a.e. in $\Omega$. 
The uniqueness follows from the fact that the functional $G_{i,n}(\cdot)$ is strictly convex. 
\end{proof}

Regarding the regularity of the minimizer $u_{i,n}$ obtained in Theorem \ref{existence-mini}, 
we start with the following: 
\begin{thm} \label{est-mini}
Let $u_{i,n}$ be the solution of \eqref{incre-mini} obtained by Theorem \ref{existence-mini}. 
Then, for any $n \in \N$,  it holds that 
\begin{gather}
\int^T_0 \!\!\! \int_\Omega \av{V_n(x,t)}^2 \, dx dt \le 2 E(u_0), \label{est-u_n-t} \\
\sup_{i}{ \Ln{\Delta u_{i,n}}{\Omega} } \le \sqrt{2 E(u_0)}. \label{u-in-uni-bd}
\end{gather}
\end{thm}
\begin{proof}
Fix $T>0$ and $n \in \N$. 
For each $i=1, \cdots, n$,  it follows from \eqref{G-func}-\eqref{to-v} and the minimality of $u_{i,n}$ that  
\begin{align} \label{Gin-bd}
G_{i, n}(u_{i,n}) \le G_{i,n}(u_{i-1, n})= E(u_{i-1, n}). 
\end{align}
Hence we get 
\begin{align*}
P_{i, n}(u_{i,n}) \le E(u_{i-1,n}) - E(u_{i,n}), 
\end{align*}
i.e., 
\begin{align} \label{pre-V-L2}
\dfrac{1}{2 \tau_n} \int_\Omega (u_{i,n} - u_{i-1,n})^2 \, dx 
 \le E(u_{i-1,n}) - E(u_{i,n}). 
\end{align}
Combining \eqref{pre-V-L2} with definitions \eqref{discrete-velocity} and \eqref{evolv-velocity}, we obtain 
\begin{align*}
\dfrac{1}{2} \int^T_0 \!\!\! \int_\Omega \av{V_n(x,t)}^2 \, dx dt 
& = \dfrac{1}{2} \sum^n_{i=1} \int^{i \tau_n}_{(i-1) \tau_n} \! \int_\Omega \av{V_{i,n}(x)}^2 \, dx dt \\
& \le \sum^n_{i=1} \left( E(u_{i-1,n}) - E(u_{i,n}) \right) 
  = E(u_0) - E(u_{n,n})
  \le E(u_0),  
\end{align*}
i.e., \eqref{est-u_n-t}. 

By \eqref{Gin-bd}, we obtain that $E(u_{i,n}) \le E(u_{i-1,n})$ for each $i=1, \cdots, n$, and then 
\begin{align} \label{pre-u-in-uni-bd}
\dfrac{1}{2} \int_\Omega (\Delta u_{i,n})^2 \, dx = E(u_{i,n}) \le E(u_0). 
\end{align}
It is clear that \eqref{pre-u-in-uni-bd} is equivalent to \eqref{u-in-uni-bd}. 
\end{proof}

By the definition of $u_{i,n}$, we see that 
\begin{align*}
&\int_{\Omega} | \Delta (u_{i,n} + \ve \vz)|^2 \, dx 
  + \dfrac{1}{2 \tau_n} \int_\Omega (u_{i,n}-u_{i-1,n}+ \ve \vz)^2 \, dx \\
& \qquad \ge \int_{\Omega} | \Delta u_{i,n}|^2 \, dx 
  + \dfrac{1}{2 \tau_n} \int_\Omega (u_{i,n}-u_{i-1,n})^2 \, dx 
\end{align*}
for any $\ve>0$ and $\vz \in H^2_0(\Omega)$ with $\vz \ge 0$. 
This implies 
\begin{align*}
\int_\Omega \Delta u_{i,n} \Delta \vz \, dx + \dfrac{1}{\tau_n} \int_\Omega (u_{i,n} - u_{i-1,n}) \vz \, dx \ge 0, 
\end{align*}
so that 
\begin{align} \label{D-measure-positive}
\mu_{i,n}:= \Delta^2 u_{i,n} + V_{i,n} \ge 0 
\end{align}
in the sense of the distribution. Hence $\mu_{i,n}$ is a measure in $\Omega$ (e.g., see \cite{S}).

Regarding the finiteness of $\mu_{i,n}$, we have the following: 

\begin{thm} \label{bilaplace-measure}
Let $u_{i,n}$ be the solution of \eqref{incre-mini} obtained by Theorem \ref{existence-mini}. 
Then $\mu_{i,n}$ defined in \eqref{D-measure-positive} is a measure in $\Omega$ for each $i=1, \cdots, n$. 
Moreover there exists a positive constant $C$ being independent of $n$ such that 
\begin{align} \label{me-bd}
\tau_n \sum^{n}_{i=1} \mu_{i,n}(\Omega)^2< C. 
\end{align}
\end{thm}
\begin{proof}
Fix $T>0$, $n \in \N$ and $i= 1, \cdots, n$ arbitrarily. 
For any $\ve>0$, we define 
\begin{align}
\gm_\ve(\lm)&= 
\begin{cases}
 \dfrac{\lm^2}{\ve} \quad & \text{if} \quad \lm<0, \\ 
 0  \quad & \text{if} \quad \lm>0, 
\end{cases} \label{def-gm} \\
\vb_\ve(\lm)&= \gm'_\ve(\lm). 
\end{align}
Let us consider the minimization problem: $\min_{v \in H^2_0(\Omega)} G^\ve_{i,n}(v)$, where 
\begin{align} \label{ve-minimum-p}
G^\ve_{i,n}(v):= \int_\Omega \left[ \dfrac{1}{2} (\Delta v)^2 + \dfrac{1}{2 \tau_n}(v-u_{i-1,n})^2 + \gm_\ve(v-f) \right]\, dx. 
\end{align}
A standard argument implies that the problem has a unique solution $w_\ve$. 
Since the variational principle yields that  for any $\vp \in H^2_0(\Omega)$ 
\begin{align*}
\int_\Omega \left[ \Delta w_\ve \Delta \vp + \dfrac{1}{\tau_n}(w_\ve - u_{i-1,n})\vp + \vb_\ve(w_\ve-f) \vp \right] \, dx=0, 
\end{align*} 
we have 
\begin{align} \label{ve-E-L-eq}
\Delta^2 w_\ve + \dfrac{1}{\tau_n}(w_\ve-u_{i-1,n})+ \vb_\ve(w_\ve-f)=0 \quad \text{in} \quad \Omega. 
\end{align}
The standard elliptic regularity theory implies that $w_\ve$ is a classical solution of \eqref{ve-E-L-eq}. 

For any $\vp \in H^2_0(\Omega)$ with $\vp \ge f$ a.e. on $\Omega$, 
the minimality of $w_\ve$ asserts that 
\begin{align} \label{w-vp}
G^\ve_{i,n}(w_\ve) \le G^\ve_{i,n}(\vp) 
 = \int_\Omega \left[ \dfrac{1}{2} (\Delta \vp)^2 + \dfrac{1}{2 \tau_n}(\vp-u_{i-1,n})^2 \right]\, dx. 
\end{align}
Since Theorem \ref{existence-mini} allows us to take $u_{i-1,n}$ as $\vp$ in \eqref{w-vp}, we have 
\begin{align} \label{w_ve-u_i-1}
G^\ve_{i,n}(w_\ve) \le \dfrac{1}{2} \int_\Omega (\Delta u_{i-1,n})^2 \, dx \le E(u_0), 
\end{align}
i.e., 
\begin{gather}
\dfrac{1}{2} \int_\Omega (\Delta w_\ve)^2 \, dx \le E(u_0), \label{w-bd-1} \\
\dfrac{1}{2 \tau_n} \int_\Omega (w_\ve -u_{i-1,n})^2 \, dx \le E(u_0), \label{w-bd-2}
\end{gather}
and 
\begin{align} \label{w-bd-gm}
\int_\Omega \gm_\ve(w_\ve -f) \, dx \le E(u_0). 
\end{align}
The inequality \eqref{w-bd-1} implies that there exist a sequence $\{ \ve' \}$ and a function $\bar{u} \in H^2_0(\Omega)$ 
such that, as $\ve' \to 0$,  
\begin{align}
& w_{\ve'} \rightharpoonup \bar{u} \quad \text{in} \quad H^2_0(\Omega), \label{w-D2-w} \\
& w_{\ve'} \to \bar{u} \quad \text{a.e. in} \quad \Omega. \label{w-a.e} 
\end{align}
By \eqref{def-gm} and \eqref{w-bd-gm}, we obtain 
\begin{align*}
\int_\Omega \av{(w_\ve - f)^-}^2 \, dx \le C \ve.  
\end{align*}
Combining \eqref{w-a.e} with Chebychev's inequality, we deduce that $(\bar{u} - f)^- =0$ a.e. in $\Omega$, 
i.e., $\bar{u} \ge f$ a.e. in $\Omega$. Thus it holds that $\bar{u} \in K$. 
In the following we shall prove that $\bar{u}$ is a minimizer of \eqref{incre-mini}, i.e., 
\begin{align*}
\min_{v \in V} \int_\Omega \left[ \dfrac{1}{2} (\Delta v)^2 + \dfrac{1}{2 \tau_n}(v- u_{i-1,n})^2  \right] \, dx. 
\end{align*}
To prove the assertion, fix $v \in K$ arbitrarily. Then we observe that 
\begin{align*}
\int_\Omega \left[ \dfrac{1}{2} (\Delta v)^2 + \dfrac{1}{2 \tau_n}(v- u_{i-1,n})^2  \right] \, dx 
 &= E(v) + P_{i,n}(v) + \int_\Omega \gm_\ve(v-f) \, dx \\
 &\ge E(w_\ve) + P_{i,n}(w_\ve)+ \int_\Omega \gm_\ve(w_\ve-f) \, dx \\
 &\ge \int_\Omega \left[ \dfrac{1}{2} (\Delta w_\ve)^2 + \dfrac{1}{2 \tau_n}(w_\ve- u_{i-1,n})^2  \right] \, dx. 
\end{align*}
Making use of \eqref{w-D2-w}-\eqref{w-a.e}, we have 
\begin{align*}
\int_\Omega \left[ \dfrac{1}{2} (\Delta v)^2 + \dfrac{1}{2 \tau_n}(v- u_{i-1,n})^2  \right] \, dx 
&\ge \liminf_{\ve' \to 0} \int_\Omega \left[ \dfrac{1}{2} (\Delta w_{\ve'})^2 
    + \dfrac{1}{2 \tau_n}(w_{\ve'}- u_{i-1,n})^2  \right] \, dx \\
&\ge \int_\Omega \left[ \dfrac{1}{2} (\Delta \bar{u})^2 + \dfrac{1}{2 \tau_n}(\bar{u} - u_{i-1,n})^2  \right] \, dx.
\end{align*}
This implies that $\bar{u}$ is a minimizer of \eqref{incre-mini}. 
Then  the uniqueness of minimizer yields $\bar{u}=u_{i,n}$. 

Recalling $\vb_\ve \le 0$, we find 
\begin{align*}
\Delta^2 w_\ve + \dfrac{1}{\tau_n}(w_\ve-u_{i-1,n})= -\vb_\ve(w_\ve-f) \ge 0, 
\end{align*}
i.e., 
\begin{align*}
\mu^\ve_{i,n}:= \Delta^2 w_\ve + \dfrac{1}{\tau_n}(w_\ve-u_{i-1,n})
\end{align*}
is a measure in $\Omega$. 
To begin with, we shall prove that $\mu^\ve_{i,n}$ converges to a measure as $\ve \to 0$ up to a subsequence. 
To do so, we claim that, for each $i$ and $n$, $\{ \mu^\ve_{i,n}(U) \}$ is uniformly bounded with respect 
to $\ve$ for any compact subset $U$ of $\Omega$. 
Indeed, for each $i$, $n$ and fixed $\psi \in C^\infty_0(\Omega)$ with $\psi \equiv 1$ on $U$ 
and $0 \le \psi < 1$ elsewhere, it follows from \eqref{w-bd-1} and \eqref{w-bd-2} that 
\begin{align} \label{mu-bd-K}
\mu^\ve_{i,n}(U) &= \int_{U} \psi d \mu^\ve_{i,n} 
\le \int_\Omega \psi d \mu^\ve_{i,n} \\ 
&= \int_\Omega \left[ \Delta w_\ve \Delta \psi + \dfrac{1}{\tau_n}(w_\ve - u_{i-1,n})\psi \right] \, dx \notag \\
&\le \left( \int_\Omega (\Delta w_\ve)^2 \,dx  \right)^{\frac{1}{2}} 
   \left( \int_\Omega (\Delta \psi)^2 \,dx  \right)^{\frac{1}{2}} \notag \\
& \qquad + \dfrac{1}{\sqrt{\tau_n}}\left( \dfrac{1}{\tau_n} \int_\Omega (w_\ve - u_{i-1,n})^2 \,dx \right)^{\frac{1}{2}} 
                           \left( \int_\Omega \psi^2 \,dx \right)^{\frac{1}{2}}.  \notag
\end{align}
Since \eqref{w_ve-u_i-1} yields that 
\begin{align} \label{E-difference}
\dfrac{1}{2 \tau_n} \int_\Omega (w_\ve - u_{i-1,n})^2 \, dx 
 &\le E(u_{i-1,n}) - E(w_\ve) - \int_\Omega \gm_\ve(w_\ve-f) \, dx \\
 &\le E(u_{i-1,n}) - E(w_\ve), \notag
\end{align}
and $\psi$ is fixed, combining \eqref{mu-bd-K} with \eqref{w-bd-1} and \eqref{E-difference}, we obtain 
\begin{align} \label{mu-bd-K-2}
\mu^\ve_{i,n}(U) \le C(U) \left[ (2 E(u_0))^{\frac{1}{2}} 
 + \left( \dfrac{E(u_{i-1,n}) - E(w_\ve)}{\tau_n} \right)^{\frac{1}{2}} \right].  
\end{align}
Then, for each $i$ and $n$,  there exist a sequence $\{ \ve'' \} \subset \{ \ve' \}$ and 
a measure $\bar{\mu}$ in $\Omega$ such that, as $\ve'' \to 0$,  
\begin{align}\label{mu-w}
\mu^{\ve''}_{i,n} \rightharpoonup \bar{\mu}, 
\end{align}
where \eqref{mu-w} means that for any function $\vz \in C_0(\Omega)$ 
\begin{align} \label{mean-mu-w}
\int_\Omega \vz d \mu^{\ve''}_{i,n} \to \int_\Omega \vz d \bar{\mu}. 
\end{align}
Furthermore, taking $\vz \in C^2_0(\Omega)$ in \eqref{mean-mu-w}, we find 
\begin{align*}
\int_\Omega \vz d \bar{\mu} 
 &= \lim_{\ve'' \to 0} \int_\Omega \left[ \Delta \vz \Delta w_{\ve''} 
     + \dfrac{1}{\tau_n} \vz (w_{\ve''} - u_{i-1,n}) \right]\, dx \\
 &=  \int_\Omega \left[ \Delta \vz \Delta \bar{u} + \dfrac{1}{\tau_n} \vz (\bar{u} - u_{i-1,n}) \right]\, dx,  
\end{align*}
so that $\bar{\mu}= \mu_{i,n}$. 

Next we shall prove that $\tau_n \sum^{n}_{i=1} \mu_{i,n}(U)$ is uniformly bounded with respect to $n$ for 
any compact set $U \subset \Omega$. 
Combining \eqref{mu-bd-K-2} with \eqref{w-D2-w} and \eqref{mu-w}, we see that  
\begin{align*}
\mu_{i,n}(U) 
  & \le C(U) \left( 2 E(u_0) \right)^{\frac{1}{2}} 
           + C(U) \liminf_{\ve \to 0} \left( \dfrac{E(u_{i-1,n}) - E(w_\ve)}{\tau_n} \right)^{\frac{1}{2}} \\
  & \le C(U) ( 2 E(u_0) )^{\frac{1}{2}} 
           + C(U) \left( \dfrac{E(u_{i-1,n}) - E(u_{i,n})}{\tau_n} \right)^{\frac{1}{2}}.  
\end{align*}
Multiplying $\tau_n$ and summing over $i=1, \cdots, n$, we obtain 
\begin{align*}
\tau_n \sum^{n}_{i=1}  \mu_{i,n}(U)^2 
 &\le C(U)' E(u_0) T + C(U)' \left[ E(u_0) - E(u_{n,n}) \right] \\
 &\le C(U)' E(u_0) (T + 1). 
\end{align*}

Finally we shall prove $\tau_n \sum^n_{i=1} \mu_{i,n}(\Omega)$ is uniformly bounded with respect to $n$. 
Multiplying the equation \eqref{ve-E-L-eq} by $w_{\ve} - f$, we find 
\begin{align} \label{ing-1-me-bd}
\int_\Omega \left[ \Delta^2 w_\ve + \dfrac{1}{\tau_n} ( w_\ve - u_{i-1,n}) \right] (w_\ve - f) \, dx 
 = - \int_\Omega \vb_\ve( w_\ve - f) (w_\ve -f) \, dx \le 0. 
\end{align}
Let $\Omega_\vd$ denote the intersection of $\Omega$ and $\vd$-neighborhood of $\pd \Omega$. 
Since $f < 0$ in $\pd \Omega$, there exists a positive constant $c$ such that 
\begin{align} \label{f-negative}
f(x) < -c \quad \text{in} \quad \Omega_\vd 
\end{align}
for $\vd>0$ small enough. From \eqref{f-negative}, we observe that 
\begin{align} \label{ing-2-me-bd}
&\int_\Omega \left[ \Delta^2 w_\ve + \dfrac{1}{\tau_n} ( w_\ve - u_{i-1,n}) \right] f \, dx \\
& \,\, \le -c \int_{\Omega_\vd} \left[ \Delta^2 w_\ve + \dfrac{1}{\tau_n} ( w_\ve - u_{i-1,n}) \right] \, dx 
 + \int_{\Omega \setminus \Omega_\vd} \left[ \Delta^2 w_\ve + \dfrac{1}{\tau_n} ( w_\ve - u_{i-1,n}) \right] f \, dx. \notag
\end{align}
On the other hand, it follows from \eqref{E-difference} and $\int_\Omega \Delta^2 w_\ve w_\ve \, dx \ge 0$ that  
\begin{align} \label{ing-3-me-bd}
\int_\Omega \left[ \Delta^2 w_\ve + \dfrac{1}{\tau_n} ( w_\ve - u_{i-1,n}) \right] w_\ve \, dx 
& \ge - \Ln{w_\ve}{\Omega} 
           \left( \dfrac{1}{\tau_n^2} \int_\Omega ( w_\ve - u_{i-1,n})^2 \, dx \right)^{\frac{1}{2}} \\
& \ge - (2 E(u_0))^{\frac{1}{2}} \left( \dfrac{E(u_{i-1,n}) - E(w_\ve)}{\tau_n} \right)^{\frac{1}{2}}. \notag
\end{align}
Then \eqref{ing-1-me-bd}, \eqref{ing-2-me-bd}, and \eqref{ing-3-me-bd} imply that 
\begin{align*}
c \int_{\Omega_\vd} d\mu^{\ve}_{i,n} 
 \le \Lp{f}{\infty}{\Omega \setminus \Omega_\vd} \int_{\Omega \setminus \Omega_\vd} d\mu^{\ve}_{i,n} 
       + (2 E(u_0))^{\frac{1}{2}} \left( \dfrac{E(u_{i-1,n}) - E(w_\ve)}{\tau_n} \right)^{\frac{1}{2}},  
\end{align*}
so that 
\begin{align*}
 \mu^{\ve}_{i,n}(\Omega_\vd) 
  \le c^{-1} \Lp{f}{\infty}{\Omega \setminus \Omega_\vd} \mu^{\ve}_{i,n}(\Omega \setminus \Omega_\vd) 
        + c^{-1} (2 E(u_0))^{\frac{1}{2}} \left( \dfrac{E(u_{i-1,n}) - E(w_\ve)}{\tau_n} \right)^{\frac{1}{2}}. 
\end{align*}
Thus we get 
\begin{align*}
 \mu^{\ve}_{i,n}(\Omega) 
  \le C_1 \mu^\ve_{i,n}(\Omega \setminus \Omega_\vd) 
  + c^{-1} (2 E(u_0))^{\frac{1}{2}} \left( \dfrac{E(u_{i-1,n}) - E(w_\ve)}{\tau_n} \right)^{\frac{1}{2}},   
\end{align*}
where $C_1= 1+ c^{-1} \Lp{f}{\infty}{\Omega \setminus \Omega_\vd}$. 
Then,  by \eqref{w-D2-w} and \eqref{mu-w} we obtain  
\begin{align*}
\mu_{i,n}(\Omega) 
  & \le C_1 \mu_{i,n}(\Omega \setminus \Omega_\vd) 
  + c^{-1} (2 E(u_0))^{\frac{1}{2}} \liminf_{\ve \to 0} 
    \left( \dfrac{E(u_{i-1,n}) - E(w_\ve)}{\tau_n} \right)^{\frac{1}{2}} \\
  & \le C_1 \mu_{i,n}(\Omega \setminus \Omega_\vd) 
  + c^{-1} (2 E(u_0))^{\frac{1}{2}} \left( \dfrac{E(u_{i-1,n}) - E(u_{i,n})}{\tau_n} \right)^{\frac{1}{2}}.  
\end{align*}
Since $\Omega \setminus \Omega_\vd$ is a compact subset of $\Omega$, 
multiplying $\tau_n$ and summing over $i=1, \cdots, n$, we observe that 
\begin{align*}
\tau_n \sum^{n}_{i=1}  \mu_{i,n}(\Omega)^2 
 &\le C_1^2 C_\vd + 2 c^{-2} E(u_0) (E(u_0) - E(u_{n,n})) \\
 &\le C_1^2 C_\vd + 2 c^{-2} E(u_0)^2, 
\end{align*}
where $C_\vd:= \tau_n \sum^{n}_{i=1} \mu_{i,n}(\Omega \setminus \Omega_\vd)^2$ is independent of $n$. 
This completes the proof. 
\end{proof}

\smallskip

In the rest of this section, we shall prove that $u_{i,n} \in W^{2,\infty}(\Omega)$ if $N \le 3$. 
In what follows, we denote the mollifier as follows: 
\begin{align*}
J_\ve (h)(x):= \int_{\R^N} j_\ve(x-y) h(y) \, dy, 
\end{align*}
where 
\begin{align*}
\quad j_\ve(x-y)= \dfrac{1}{\ve^n} j\left(\dfrac{x-y}{\ve} \right)
\end{align*}
and the function $j(x)=j_0(|x|)$ satisfies 
\begin{align*}
j_0 \in C^\infty(\R), \quad j_0(t)=0 \,\,\,\, \text{if} \,\,\,\, |t|>1, \quad  
j_0(t) \ge 0, \quad \int_{\R^N} j_0(|x|) \, dx=1. 
\end{align*}

Here we show a property of the support of $\mu_{i,n}$.  

\begin{lem} \label{supp-mu-i-n-1}
Let $x_{0} \in \Omega$. Assume that there exist a neighborhood $W$ of $x_{0}$ and a constant 
$\vd>0$ such that 
\begin{align} \label{cond-measure0}
J_{\ve} (u_{i,n})(x) - f(x) > \vd \quad \text{in} \quad W. 
\end{align}
Then $\mu_{i,n}=0$ in $W$. 
\end{lem}
\begin{proof}
We extend $u_{i,n} \in H^{2}_{0}(\Omega)$ to become a function in $H^{2}(\R^{n})$. 
By the assumption \eqref{cond-measure0}, it holds that $u_{i,n} \pm \vz \in K$ for any 
$\vz \in C^{\infty}_{c}(W)$ with $\av{\vz} < \vd$. 
Since $u_{i,n}$ is the unique minimizer of \eqref{incre-mini}, one can verify that 
for any $\vz \in C^{\infty}_{c}(W)$ with $\av{\vz} < \vd$
\begin{align} \label{variation-1}
& \dfrac{1}{2} \int_{\Omega} \av{\Delta J_{\ve}(u_{i,n})}^{2} \, dx
 + \dfrac{1}{2 \tau_{n}} \int_{\Omega} (J_{\ve}(u_{i,n}) - u_{i-1,n} )^{2} \, dx \\
& \qquad \qquad 
\le \dfrac{1}{2} \int_{\Omega} \av{\Delta J_{\ve}(u_{i,n}) \pm \Delta \vz}^{2} \, dx 
 + \dfrac{1}{2 \tau_{n}} \int_{\Omega} (J_{\ve}(u_{i,n}) \pm \vz - u_{i-1,n} )^{2} \, dx.  \notag
\end{align}
Letting $\ve \downarrow 0$ in \eqref{variation-1}, we find 
\begin{align*} 
& \dfrac{1}{2} \int_{\Omega} \av{\Delta u_{i,n}}^{2} \, dx
 + \dfrac{1}{2 \tau_{n}} \int_{\Omega} (u_{i,n} - u_{i-1,n} )^{2} \, dx \\
& \qquad \qquad 
\le \dfrac{1}{2} \int_{\Omega} \av{\Delta u_{i,n} \pm \Delta \vz}^{2} \, dx 
 + \dfrac{1}{2 \tau_{n}} \int_{\Omega} (u_{i,n} \pm \vz - u_{i-1,n} )^{2} \, dx,   \notag
\end{align*}
so that 
\begin{align} \label{variation-2}
0 \le \pm \left( \int_{\Omega} \left\{ \Delta u_{i,n} \Delta \vz + V_{i,n} \vz \right\} \, dx \right) 
 + \dfrac{1}{2} \int_{\Omega} \av{\Delta \vz}^{2} \, dx + \dfrac{1}{2 \tau_{n}} \int_{\Omega} \vz^{2} \, dx 
\end{align}
for any $\vz \in C^{\infty}_{c}(W)$ with $\av{\vz} < \vd$. 
Fix $\vz \in C^{\infty}_{c}(W)$ with $\av{\vz} < \vd$ arbitrarily. 
Then we asserts from \eqref{variation-2} that 
\begin{align} \label{variation-3} 
0 \le \pm \ve \left( \int_{\Omega} \left\{ \Delta u_{i,n} \Delta \vz + V_{i,n} \vz \right\} \, dx \right) 
 + \dfrac{\ve^{2}}{2} \int_{\Omega} \av{\Delta \vz}^{2} \, dx + \dfrac{\ve^{2}}{2 \tau_{n}} \int_{\Omega} \vz^{2} \, dx.   
\end{align}
Since $\mu_{i,n} \ge 0$, it follows from \eqref{variation-3} that 
\begin{align*}
0 \le \dfrac{\int_{\Omega} \left\{ \Delta u_{i,n} \Delta \vz +  V_{i,n} \vz \right\} \, dx}
 {\frac{1}{2} \int_{\Omega} \av{\Delta \vz}^{2} \, dx + \frac{1}{2 \tau_{n}} \int_{\Omega} \vz^{2} \, dx} \le \ve. 
\end{align*}
Since $\ve>0$ is arbitral, this inequality implies that 
\begin{align*}
\int_{\Omega} \left\{ \Delta u_{i,n} \Delta \vz + V_{i,n} \vz \right\} \, dx=0. 
\end{align*}
This completes the proof. 
\end{proof}

\smallskip

We denote the inverse operator of the Laplacian by $\Delta^{-1}$, i.e., if $w$ satisfies 
\begin{align*}
\begin{cases}
-\Delta w = g \quad & \text{in} \quad \Omega, \\
w=0 \quad & \text{on} \quad \pd \Omega, 
\end{cases}
\end{align*}
then we write $\Delta^{-1} g= w$. We note that the estimate 
\begin{align} \label{elliptic}
\Hn{\Delta^{-1} g}{2} \le C \Ln{g}{\Omega}
\end{align}
is followed from the elliptic regularity (e.g., see \cite{GT}). 

We start with the following lemma: 

\begin{lem} \label{C-F-lem-1} 
For each $n \in \N$ and $i\in\{1, \cdots, n\}$, there exists a function $v_{i,n}$ satisfying the following
properties{\rm :}
\begin{enumerate}
\item[{\rm (a)}] $v_{i,n}= \Delta u_{i,n}+ \Delta^{-1} V_{i,n}$ a.e. in $\Omega${\rm ;} 
\item[{\rm (b)}] $v_{i,n}$ is upper semicontinuous in $\Omega${\rm ;}
\item[{\rm (c)}] For any $x^0 \in \Omega$ and for any sequence of balls $B_\rho(x^0)$ with center $x^0$ 
and radius $\rho$, it holds that 
\begin{align}
\dfrac{1}{\av{B_\rho(x^0)}} \int_{B_\rho(x^0)} v_{i,n} \, dx \downarrow v_{i,n}(x^0) 
 \quad \text{as} \quad \rho \downarrow 0. 
\end{align}
\end{enumerate}
\end{lem}
\begin{proof}
Let us define 
\begin{align*}
v^\rho_{i,n}(x)= \dfrac{1}{\av{B_\rho(x)}} \int_{B_\rho(x)} 
                  \left[ \Delta u_{i,n}(y) + \Delta^{-1} V_{i,n}(y) \right] \, dy. 
\end{align*}
We claim that, for any $x^0 \in \Omega$, $v^\rho_{i,n}(x^0)$ is decreasing in $\rho$. 
Indeed, if $u_{i,n} \in C^\infty(\Omega)$, we obtain from Green's formula that 
\begin{align*}
\Delta u_{i,n}(x^0) + \Delta^{-1} V_{i,n}(x^0) 
& = \dfrac{1}{\av{\pd B_\rho(x^0)}} \int_{\pd B_{\rho}(x^0)} 
    \left[ \Delta u_{i,n} + \Delta^{-1} V_{i,n} \right] \, dS \\
& \qquad - \int_{B_\rho(x^0)} \left[ \Delta^2 u_{i,n}(x) + V_{i,n}(x) \right] G_\rho(x-x_0) \, dx,
\end{align*}
where $G_\rho$ is Green's function given by 
\begin{align} \label{def-G}
G_\rho(r)= 
\begin{cases}
\vspace{0.1cm}
\dfrac{1}{2}(r-\rho) \quad & \text{if} \quad N=1, \\
\vspace{0.1cm}
\dfrac{1}{2\pi} \log{\dfrac{\rho}{r}} \quad & \text{if} \quad N=2, \\
\dfrac{1}{N(N-2) \omega(N)} (r^{2-N} - \rho^{2-N}) \quad & \text{if} \quad N \ge 3. 
\end{cases}
\end{align}
Remark that $\omega(N)$ denotes the volume of unit ball in $\R^N$. 
From \eqref{D-measure-positive} and $G_{\rho'} > G_\rho$ if $\rho' > \rho$, we get 
\begin{align*}
\dfrac{1}{\av{\pd B_\rho(x^0)}} \int_{\pd B_\rho(x^0)} \left[ \Delta u_{i,n} + \Delta^{-1} V_{i,n} \right] \, dS
 \le \dfrac{1}{\av{\pd B_{\rho'}(x^0)}} 
    \int_{\pd B_{\rho'}(x^0)} \left[ \Delta u_{i,n} + \Delta^{-1} V_{i,n} \right] \, dS, 
\end{align*}
and, by integration, 
\begin{align} \label{average-1}
&\dfrac{1}{\av{B_\rho(x^0)}} \int_{B_\rho(x^0)} 
    \left[ \Delta u_{i,n}(x) + \Delta^{-1} V_{i,n}(x) \right] \, dx \\
& \qquad \le \dfrac{1}{\av{B_{\rho'}(x^0)}} \int_{B_{\rho'}(x^0)} 
    \left[ \Delta u_{i,n}(x) + \Delta^{-1} V_{i,n}(x) \right] \, dx. \notag
\end{align}

For general $u_{i,n} \in H^2_0(\Omega)$ with \eqref{D-measure-positive}, we introduce the $C^\infty$ functions 
\begin{align*}
U_m := J_{1/m}(\Delta u_{i,n} + \Delta^{-1} V_{i,n}).  
\end{align*}
Since $\Delta U_m \ge 0$, we can deduce from \eqref{average-1} that 
\begin{align*}
\dfrac{1}{\av{B_\rho(x^0)}} \int_{B_\rho(x^0)} U_m \, dx 
\le \dfrac{1}{\av{B_{\rho'}(x^0)}} \int_{B_{\rho'}(x^0)} U_m \, dx.
\end{align*} 
Letting $m \to +\infty$, we obtain \eqref{average-1} for general $u_{i,n} \in H^{2}_{0}(\Omega)$.  
Thus we conclude that 
\begin{align}
v^\rho_{i,n}(x) \downarrow v_{i,n}(x) \quad \text{as} \quad \rho \downarrow 0, 
\end{align}
where $v_{i,n}$ is a some function. 

Since $v^\rho_{i,n}$ is continuous in $x$, we see that $v_{i,n}$ is upper semicontinuous. 
Recalling that $\Delta u_{i,n} + \Delta^{-1} V_{i,n} \in L^2(\Omega)$, we also obtain that, 
as $\rho \downarrow 0$, 
\begin{align*}
v^\rho_{i,n} \to \Delta u_{i,n} + \Delta^{-1} V_{i,n} \quad \text{a.e. in} \quad \Omega. 
\end{align*}
Consequently we have 
\begin{align*}
v_{i,n}= \Delta u_{i,n} + \Delta^{-1} V_{i,n} \quad \text{a.e. in} \quad \Omega.  
\end{align*}
This completes the proof. 
\end{proof}

\begin{lem} \label{C-F-lem-2}
Let $1 \le N \le 7$, 
then for any point $x^0 \in \Omega$ that belongs to 
the support of $\mu_{i,n}$, it holds that 
\begin{align}
v_{i,n}(x^0) - \Delta^{-1} V_{i,n}(x^0) \ge \Delta f(x^0) 
\end{align}
for each $n \in \N$ and $i=1, \cdots, n$. 
\end{lem}

\begin{proof}
With the aid of Lemma \ref{supp-mu-i-n-1}, we asserts that ${\rm supp} \mu_{i,n}$ is contained in the set of points 
where \eqref{cond-measure0} is not satisfies.  
Thus, if $x^0 \in {\rm supp} \, \mu_{i,n}$, 
then there exist sequences $x_m \to x^0$ and $\ve_m \downarrow 0$ such that 
\begin{align} \label{to-coincide}
(J_{\ve_m} u_{i,n})(x_m) - f(x_m) \to 0. 
\end{align}
By Green's formula, we have 
\begin{align} \label{Green-u}
(J_\ve u_{i,n})(x_m)= \dfrac{1}{\av{S_{\rho,m}}} \int_{S_{\rho,m}} J_\ve u_{i,n} \, dS 
    - \int_{B_{\rho,m}} \Delta (J_\ve u_{i,n})(y) G_\rho(x_m-y) \, dy, 
\end{align}
where $B_{\rho,m} := \{ \av{y-x_m} < \rho \}$, $S_{\rho,m} := \pd B_{\rho,m}$.  
Similarly it holds that 
\begin{align} \label{Green-f}
(J_\ve f)(x_m)= \dfrac{1}{\av{S_{\rho,m}}} \int_{S_{\rho,m}} J_\ve f \, dS 
    - \int_{B_{\rho,m}} \Delta (J_\ve f)(y) G_\rho(x_m-y) \, dy.  
\end{align}
Since it follows from $u_{i,n} \ge f$, also $J_\ve u \ge J_\ve f$, that 
\begin{align*}
\dfrac{1}{\av{S_{\rho,m}}} \int_{S_{\rho,m}} J_\ve u_{i,n} \, dS 
 \ge \dfrac{1}{\av{S_{\rho,m}}} \int_{S_{\rho,m}} J_\ve f \, dS,   
\end{align*}
using the inequality and \eqref{to-coincide}, we obtain, by comparing \eqref{Green-u} with \eqref{Green-f}, that 
\begin{align} \label{liminf-positive}
\liminf_{m \to +\infty} \left[ \int_{B_{\rho,m}} \!\!\!\!\!\!\! \Delta (J_\ve u_{i,n})(y) G_\rho(x_m-y) \, dy
     - \int_{B_{\rho,m}} \!\!\!\!\!\!\! \Delta (J_\ve f)(y) G_\rho(x_m-y) \, dy \right] \ge 0.  
\end{align}
Using a change of variables and integrating by parts, we can reduce the first term in \eqref{liminf-positive} to 
\begin{align} \label{change}
\int_{B_{\rho,m}} \!\!\!\!\!\! \Delta (J_\ve u_{i,n})(y) \cdot G_\rho(x_m-y) \, dy 
 = \int_{B_{\rho,m}} \!\!\!\!\! (J_\ve G_\rho)(x_m-y) \Delta u_{i,n}(y) \, dy + \lm_{\ve,m}, 
\end{align}
where 
\begin{align*}
\lm_{\ve,m}:= -\int_{B_{\rho+\ve,m} \setminus B_{\rho,m}} \!\!\!\!\!\!\!\!\!\!\!\!\!\!\!\!\!\!\!\!
                   G_\rho (x_m-y) \Delta (J_\ve u_{i,n})(y) \, dy 
           + \int_{B_{\rho+\ve,m}}\!\!\!\!\!\!\!\!\! G_\rho(x_m-y) 
           \int_{B_{\rho+ 2 \ve,m} \setminus B_{\rho,m}} \!\!\!\!\!\!\!\!\!\!\!\!\!\!\!\!\!\!\!\!\!\! j_\ve (y-z) 
                  \Delta u_{i,n}(z) \, dy
\end{align*}
and $\lm_{\ve,m} \to 0$ as $\ve \downarrow 0$ uniformly in  $m$. 
A similar relation holds for the second integral in \eqref{liminf-positive}. 
Therefore we obtain 
\begin{align} \label{liminf-positive-1}
\liminf_{m \to +\infty} \int_{B_{\rho,m}} \!\!\!\!\! (J_{\ve_m} G_\rho)(x_m-y) 
                                [v_{i,n}(y) - \Delta^{-1} V_{i,n}(y) - \Delta f(y)] dy \ge 0. 
\end{align}
Recalling that $V_{i,n} \in H^{2}_{0}(\Omega)$ for each $n \in \N$, 
we see that $\Delta^{-1} V_{i,n} \in H^{4}(\Omega)$ by the elliptic regularity (see \cite{GT}). 
Then it follows from Sobolev's embedding that $\Delta^{-1}V_{i,n}$ is continuious in $\Omega$ for $1 \le N \le 7$. 
Furthermore since $v_{i,n}$ is upper semicontinuous, there exists a point $x_{m,\rho} \in \con{B}_{\rho,m}$ 
such that the maximum of the function $v_{i,n}(x) - \Delta^{-1} V_{i,n}(x) - \Delta f(x)$ in $ \con{B}_{\rho,m}$ 
attains at $x= x_{m,\rho}$. 
Then \eqref{liminf-positive-1} implies that 
\begin{align*}
v_{i,n}(x_{m,\rho}) - \Delta^{-1} V_{i,n}(x_{m,\rho}) -\Delta f(x_{m,\rho}) 
 \ge -\vd_m, \quad \vd_m \to 0 \quad \text{as} \quad m \to +\infty. 
\end{align*}
We may assume that $x_{m,\rho} \to x_\rho$ for some $x_\rho \in \{ y \in \R^N :\  | y-x^0| \le \rho\}$, 
for the sequence $\{ x_{m,\rho}\}$ is bounded. 
By the upper semicontinuity of $v_{i,n}$, as $m \to +\infty$, it holds that  
\begin{align*}
v_{i,n}(x_\rho) - \Delta^{-1} V_{i,n}(x_\rho) - \Delta f(x_\rho) \ge 0. 
\end{align*}
Letting $\rho \to 0$ and using again the upper semicontinuity of $v_{i,n}$, we see that $x_\rho \to x^0$ and 
\begin{align*}
v_{i,n}(x_0) - \Delta^{-1} V_{i,n}(x^0) - \Delta f(x^0) \ge 0. 
\end{align*}
\end{proof}

Making use of Lemmas \ref{C-F-lem-1} and \ref{C-F-lem-2}, we can obtain a local bound of $\Delta u_{i,n}$: 

\begin{lem} \label{C-F-lem-3}
Let $N \le 3$. 
It holds that 
\begin{align}
\Delta u_{i,n} \in L^\infty_{loc}(\Omega)  
\end{align}
for each $n \in \N$ and $i=1, \cdots, n$. Moreover, for any $R>0$ with $\con{B}_R \subset \Omega$, 
there exist positive constants $C_1$, $C_2$, and $C_3$ being independent of $i$ and $n$ such that 
\begin{align} \label{uin-D2i}
\Li{\Delta u_{i,n}}{B_{R/3}} 
 \le C_1 E(u_0)^{\frac{1}{2}} + C_2 \Ln{V_{i,n}}{\Omega} 
      + C_3 \mu_{i,n}(D_{R/2}) + \Li{\Delta f}{B_{R/2}},  
\end{align}
where $D_{R/2}:= B_R \setminus B_{R/2}$. 
\end{lem}
\begin{proof}
Set 
\begin{align} \label{def-U}
U_{i,n}:= u_{i,n} + (\Delta^2)^{-1} V_{i,n},  
\end{align}
where $(\Delta^2)^{-1} V_{i,n}$ denotes a unique solution of 
\begin{align*}
\begin{cases}
-\Delta w= \Delta^{-1} V_{i,n} \quad & \text{in} \quad \Omega, \\
w=0 \quad & \text{on} \quad \pd \Omega. 
\end{cases}
\end{align*}
Let fix $x^0 \in \Omega$ arbitrarily and denote by $B_\rho$ the ball with center $x^0$ and radius $\rho$. 
Choose $R>0$ such that $\con{B}_R \subset \Omega$ and $\vz \in C^\infty_0(B_R)$, $\vz =1$ in $B_{2R/3}$, 
$0 \le \vz \le 1$ elsewhere. 
For any $x \in B_{2R/3}$, we have 
\begin{align*}
\Delta (J_\ve U_{i,n})(x)= \Delta (J_\ve U_{i,n})(x) \vz(x) 
 = - \int_{B_R} G_R(x-y) \Delta (\Delta(J_\ve U_{i,n}) \vz)(y) \, dy, 
\end{align*}
where $G_R$ is Green's function defined in \eqref{def-G}. 
Expanding the right-hand side, we obtain 
\begin{align} \label{rep-DJu}
\Delta(J_\ve U_{i,n})(x) 
& = -\int_{B_{R/2}} \!\!\!\!\!\! G_R(x-y) \Delta^2(J_\ve U_{i,n})(y) \, dy \\
& \qquad -\int_{D_{R/2}} \!\!\!\!\!\!\! G_R(x-y) \Delta^2 (J_\ve U_{i,n})(y) \vz(y)\, dy + \va_\ve(x), \notag   
\end{align}
where $D_{R/2}:= B_R \setminus B_{R/2}$ and 
\begin{align*}
\va_\ve(x) &:= - 2\int_{D_{R/2}} \!\!\!\!\!\!\! G_R(x-y) \nabla (\Delta(J_\ve U_{i,n}))(y) \cdot \nabla \vz(y)\, dy \\
      & \qquad - \int_{D_{R/2}} \!\!\!\!\!\!\! G_R(x-y) \Delta(J_\ve U_{i,n})(y) \Delta \vz(y)\, dy
                   := \va_{\ve,1}(x) + \va_{\ve,2}(x). 
\end{align*}

Noticing that ${\rm supp}\, \nabla \vz$ is contained in $D_{R/3}:= B_R \setminus B_{2R/3}$, we get 
\begin{align*}
\va_{\ve,1}(x) = -\int_{D_{R/3}} \!\!\! \Delta(J_\ve U_{i,n})(y) \nabla \cdot (G_R(x-y) \nabla \vz(y))\, dy. 
\end{align*}
Since the fact that $u_{i,n} \in H^{2}_{0}(\Omega)$ implies 
\begin{align*} 
\int_\Omega \av{\Delta (J_\ve U_{i,n})(y)}^2 dy \le \Ln{\Delta U_{i,n}}{\Omega}^2, 
\end{align*}
the terms $\va_{\ve,1}(x)$ and $\va_{\ve,2}(x)$ are estimated for any $x \in B_{2R/3}$ as follows: 
\begin{align*}
\av{\va_{\ve,1}(x)} & \le C \Ln{\Delta U_{i,n}}{\Omega} ( \Ln{\nabla \vz}{D_{R/3}} + \Ln{\Delta \vz}{D_{R/3}}); \\
\av{\va_{\ve,2}(x)} &\le C \Ln{\Delta U_{i,n}}{\Omega} \Ln{\Delta \vz}{D_{R/3}}.  
\end{align*}
Thus we deduce that 
\begin{align} \label{va-ve-bdd}
\av{\va_\ve(x)} \le C \Ln{\Delta U_{i,n}}{\Omega} \quad \text{for all} \quad x \in B_{2R/3}, 
\end{align}
where the constant $C$ is independent of $\ve$, $i$, and $n$. 

Along the same line as in \eqref{change}, the first term in the right-hand side of \eqref{rep-DJu} is reduced to 
\begin{align} \label{switch}
\int_{B_{R/2}} G_R(x-y) \Delta^2 (J_\ve U_{i,n})(y) \, dy 
 = \int_{B_{R/2}} (J_\ve G_R)(x-y) \Delta^2 U_{i,n}(y)\, dy + \vb_\ve(x),  
\end{align}
where $\vb_\ve(x) \to 0$ as $\ve \downarrow 0$ if $x \in B_{R/2}$. 

Consider the integral 
\begin{align*}
\tilde{G}_R(x):= \int_{B_{R/2}} \!\!\!\!\!\! G_R(x-y) d \mu_{i,n}(y). 
\end{align*}
The integral is well defined in the sense of improper integrals, that is, as 
\begin{align*}
\lim_{\vd \to 0} \int_{\{ y \in B_{R/2} :\  |x-y|>\vd\} } \!\!\!\!\!\!\!\!\!\!\!\!\!\!\!\!\!\!\!\!\!\!
G_R(x-y) d\mu_{i,n}(y) \quad \text{for a.e.} \quad x. 
\end{align*}
Indeed, this follows from Fubini's theorem since for any $k < +\infty$ it holds that  
\begin{align*}
\int_{B_{R/2}} \! \int_{|x| < k} G_R(x-y) \, dx d\mu_{i,n}(y) 
 \le C \int_{B_{R/2}} d \mu_{i,n}(y) < +\infty. 
\end{align*}
Moreover one can verify that $\tilde{G}_R$ is a superharmonic function (e.g., see \cite{L}). 

Since $G_R(z)$ is harmonic if $|z|>\ve$, one can verify that $(J_\ve G_R)(z)= G_R(z)$ holds for $|z|>\ve$. 
On the other hand, from 
\begin{align*}
(J_\ve G_R)(z) 
 = \int_{\av{y-z}<\ve} j_\ve(y-z) G_R(y) \, dy 
 = \int_{\av{\vz}<1} j_0(\vz) G_R(z+ \ve \vz) \, d\vz 
 \le C, 
\end{align*}
we see that there exists an $\ve>0$ small enough such that $(J_\ve G_R)(z) \le G_R(z)$ for $|z| < \ve$. 
Therefore Lubesgue's convergence theorem gives us that  
\begin{align} \label{Gve-tG}
\lim_{\ve \downarrow 0} \int_{B_{R/2}} (J_\ve G_R)(x-y) d \mu_{i,n}(y) 
 = \tilde{G}_R(x) \quad \text{for a.e.} \quad x \in B_{R/2}. 
\end{align}
Analogously to \eqref{switch} we have, for $x \in B_{R/2}$, 
\begin{align*}
\int_{B_R \setminus B_{R/2}} \!\!\!\!\!\!\!\!\!\!\!\!\!\! G_R(x-y) \Delta^2(J_\ve U_{i,n})(y) \vz(y) \, dy 
 = \int_{B_R \setminus B_{R/2}} \!\!\!\!\!\!\!\!\!\!\!\!\!\! J_\ve (\vz(y) G_R(x-y)) \Delta^2 U_{i,n}(y) \, dy 
    + \tilde{\vb}_\ve(x),  
\end{align*} 
where $\tilde{\vb}_\ve(x) \to 0$ as $\ve \downarrow 0$. 
Thus we deduce from Lebesgue's convergence theorem that for $x \in B_{R/2}$, as $\ve \downarrow 0$, 
\begin{align} \label{Guve-to-Gu}
\int_{B_R \setminus B_{R/2}} \!\!\!\!\!\!\!\!\!\!\!\!\!\! G_R(x-y) \Delta^2(J_\ve U_{i,n})(y) \vz(y) \, dy  
 \to \int_{B_R \setminus B_{R/2}} \!\!\!\!\!\!\!\!\!\!\!\!\!\! G_R(x-y) \Delta^2 U_{i,n}(y) \vz(y) \, dy.  
\end{align} 
We can write 
\begin{align*}
\Delta (J_\ve U_{i,n})(x)
 &= \int_{\av{z-x}<\ve} U_{i,n}(z) \Delta j_\ve(x-z) \, dz 
 = \int_{\av{z-x}<\ve} \Delta U_{i,n}(z) j_\ve(x-z) \, dz \\
 &= \int_{\av{z-x}<\ve} v_{i,n}(z) j_\ve(x-z) \, dz 
 = \int^\ve_0 \dfrac{1}{\ve^N} j_0 \left( \dfrac{\rho}{\ve} \right) 
    \int_{\pd B_\rho(x)} v_{i,n}(\rho,\vs) \, dS_\vs d\rho, 
\end{align*}
where $(\rho,\vs)$ is the spherical coordinates about $x$ and $\lm_\ve(\rho)$ is 
a smooth nonnegative function. 
Since it follows from the proof of Lemma \ref{C-F-lem-1} that 
\begin{align*}
\dfrac{1}{\av{\pd B_\rho(x)}} \int_{\pd B_\rho(x)} v_{i,n}(\rho,\vs) \, dS_\vs 
  \downarrow v_{i,n}(x) \quad \text{as} \quad \rho \downarrow0, 
\end{align*}
the mean value theorem yields that 
\begin{align*}
\Delta (J_\ve U_{i,n})(x) 
 &= \dfrac{1}{\av{\pd B_{\rho'}}} \int_{\pd B_{\rho'}} v_{i,n}(\rho',\vs) \, dS_\vs 
       \int^\ve_0 \dfrac{1}{\ve^N} j_0 \left( \dfrac{\rho}{\ve} \right) \omega_N \rho^{N-1} d\rho \\
 &= \dfrac{1}{\av{\pd B_{\rho'}} }\int_{\pd B_{\rho'}} v_{i,n}(\rho',\vs) \, dS_\vs \int_{\av{y}<1} j_0 ( \av{y} ) dy \\
 &= \dfrac{1}{\av{\pd B_{\rho'}}} \int_{\pd B_{\rho'}} v_{i,n}(\rho',\vs) \, dS_\vs 
   \downarrow v_{i,n}(x) \quad \text{as} \quad \ve \downarrow 0, 
\end{align*}
where $\omega_N \rho^{N-1}$ denotes the area of surface $\pd B_\rho$ and $\rho' \in (0, \ve)$.  
Combining this with \eqref{switch}, \eqref{Gve-tG}, and \eqref{Guve-to-Gu}, letting $\ve \downarrow 0$ in \eqref{rep-DJu}, 
we obtain that for $x \in B_{R/2}$ there holds 
\begin{align} \label{v-tG-1}
v_{i,n}(x) = -\tilde{G}_R(x)- \int_{D_{R/2}} \!\!\!\!\!\!\! \vz(y) G_R(x-y) \Delta^2 U_{i,n}(y) \, dy + \vd(x). 
\end{align} 
Remark that \eqref{va-ve-bdd} implies  
\begin{align} \label{vd-bdd}
\av{\vd(x)} \le C_1 \Ln{\Delta U_{i,n}}{\Omega} \quad \text{for all} \quad x \in B_{2R/3}, 
\end{align}
where the constant $C_1$ is independent of $i$ and $n$. 
Recalling that $\tilde{G}_R$ is superharmonic, we shall apply a maximal principle for 
superharmonic functions to $\tilde{G}_R$. 
It follows from Lemma \ref{C-F-lem-2} that 
\begin{align*}
v_{i,n}(x) \ge \Delta^{-1} V_{i,n}(x) + \Delta f(x) \quad \text{on} \quad {\rm supp}\,\mu_{i,n} \lfloor B_{R/2}. 
\end{align*}
Since the integral on the right-hand side of \eqref{v-tG-1} is non-negative, we see that 
\begin{align} \label{pre-tG-bdd}
\tilde{G}_R(x) 
 &\le -v_{i,n}(x) + \vd(x) 
 \le - \Delta^{-1} V_{i,n}(x) - \Delta f(x) + \vd(x) \\
 &\le \Cn{\Delta^{-1} V_{i,n}}{}{B_{R/2}} + \Li{\Delta f}{B_{R/2}} + \Li{\vd}{B_{R/2}}
\quad \text{on} \quad {\rm supp}\, \mu_{i,n} \lfloor B_{R/2}.  \notag
\end{align}
Furthermore Proposition \ref{compactness} and \eqref{elliptic} assert that 
\begin{align} \label{Vin-bdd}
\Cn{\Delta^{-1} V_{i,n}}{}{B_{R/2}} 
 \le \| \Delta^{-1} V_{i,n} \|_{C^{k,\gm}(\Omega)} 
 \le C \Hn{\Delta^{-1} V_{i,n}}{2} 
 \le C_2 \Ln{V_{i,n}}{\Omega}, 
\end{align}
where $k=1$ and $0 < \gm < 1/2$ if $N=1$, $k=0$ and $0 < \gm < 2-N/2$ if $N=2$, $3$, 
and the constant $C_2$ is independent of $i$ and $n$. 
Thus, combining \eqref{pre-tG-bdd} with \eqref{vd-bdd} and \eqref{Vin-bdd}, we observe that 
\begin{align*}
\tilde{G}_R(x) \le C_1 \Ln{\Delta U_{i,n}}{(\Omega} + C_2 \Ln{V_{i,n}}{\Omega} + \Li{\Delta f}{B_{R/2}} 
 \quad \text{on} \quad {\rm supp}\, \mu_{i,n} \lfloor B_{R/2}, 
\end{align*}
and then, Theorems 1.5 and 1.6 in \cite{L} give us that 
\begin{align*}
\tilde{G}_R(x) \le C_1 \Ln{\Delta U_{i,n}}{(\Omega} 
  + C_2 \Ln{V_{i,n}}{\Omega} + \Li{\Delta f}{B_{R/2}} \quad \text{in} \quad \R^N.
\end{align*}
Observing that the integral in \eqref{v-tG-1} is estimated as 
\begin{align*}
\int_{D_{R/2}} \!\!\!\!\!\!\! \vz(y) G_R(x-y) \Delta^2 U_{i,n}(y) \, dy 
 \le C_3 \mu_{i,n}(D_{R/2})  
\quad \text{in} \quad B_{R/3}, 
\end{align*}
we deduce that, for any $x \in B_{R/3}$, 
\begin{align*}
\av{v_{i,n}(x)} \le 2 C_1 \Ln{\Delta U_{i,n}}{(\Omega} + C_2 \Ln{V_{i,n}}{\Omega} 
                                + C_3 \mu_{i,n}(D_{R/2}) + \Li{\Delta f}{B_{R/2}},
\end{align*}
so that 
\begin{align} \label{pre-W2i-1}
\av{\Delta u_{i,n}(x)} \le 2 C_1 \Ln{\Delta U_{i,n}}{(\Omega} + 2 C_2 \Ln{V_{i,n}}{\Omega} 
                                  + C_3 \mu_{i,n}(D_{R/2}) + \Li{\Delta f}{B_{R/2}}. 
\end{align}
Since \eqref{u-in-uni-bd} yields that 
\begin{align*}
\Ln{\Delta U_{i,n}}{\Omega} 
 \le \Ln{\Delta u_{i,n}}{\Omega} + \Ln{\Delta^{-1} V_{i,n}}{\Omega} 
 \le \sqrt{2 E(u_0)} + C \Ln{V_{i,n}}{\Omega}, 
\end{align*}
we obtain  
\begin{align*} 
\Li{\Delta u_{i,n}}{B_{R/3}}  
 \le C_1' \sqrt{2 E(u_0)} + C_2' \Ln{V_{i,n}}{\Omega}
       + C_3 \mu_{i,n}(D_{R/2}) + \Li{\Delta f}{B_{R/2}}. 
\end{align*}
This completes the proof. 
\end{proof}

\begin{rem}\label{rem}\rm
We need to impose the restriction on the dimension $N\le 3$ 
in Lemma \ref{C-F-lem-3} in order to obtain the inequality 
\begin{align*}
\| \Delta^{-1} V_{i,n} \|_{L^{\infty}(B_{R/2})} \le C \| V_{i,n} \|_{L^{2}(\Omega)}
\end{align*}
in \eqref{Vin-bdd}. Such an estimate will allow us to prove  
a uniform $W^{2,\infty}$ bound on $u_{i,n}$ with respect to $n$. 
\end{rem}

\begin{thm} \label{u-W2infinity}
Let $N \le 3$. It holds that 
\begin{align}
u_{i,n} \in W^{2, \infty}(\Omega) 
\end{align}
for each $n \in \N$ and $i=1, \cdots, n$. Moreover, for any $R>0$ with $\con{B}_R \subset \Omega$, 
there exist positive constants $C_1$ and $C_2$ being independent of $n$ such that 
\begin{align} \label{est-W2infinity}
\tau_n \sum^n_{i=1} \Li{D^2 u_{i,n}}{\Omega}^2 
 \le C_1 + C_2 \Li{\Delta f}{\Omega}^2. 
\end{align}
\end{thm}

\begin{proof}
Thanks to Theorem \ref{est-mini}, we see that $u_{i,n}$ is uniformly bounded in $H^2_0(\Omega)$. 
Then, Proposition \ref{compactness} asserts that $u_{i,n}$ is also uniformly bounded in $C^{1, \gm}(\Omega)$ with $0 < \gm < 1/2$ 
if $N=1$, and in $C^{0,\gm}(\Omega)$ with $\gm \in (0, 2-N/2)$ if $N=2$, $3$. 
Since $u_{i,n}=0$ on $\pd \Omega$, there exists a neighborhood $\Omega_\vd$ of 
$\pd \Omega$ such that $u_{i,n} > f$ in $\Omega_\vd$. 
By the standard elliptic regularity theory, we observe that $\Delta u_{i,n} \in H^2(\Omega_\vd)$ with 
\begin{align} \label{ing-3-4-1}
\| \Delta u_{i,n} \|_{H^2(\Omega_\vd)} \le C (\Ln{V_{i,n}}{\Omega_\vd} + \Ln{\Delta u_{i,n}}{\Omega_\vd}), 
\end{align}
where the positive constant $C$ depends only on $\Omega_\vd$. Combining \eqref{ing-3-4-1} with 
the interpolation inequality 
\begin{align*}
\Li{\Delta u_{i,n}}{\Omega_\vd} \le K \| \Delta u_{i,n} \|_{H^2(\Omega_\vd)}^{N/4} \Ln{\Delta u_{i,n}}{\Omega_\vd}^{1-N/4}, 
\end{align*}
where $K$ is a positive constant depending only on $N$, we deduce that  
\begin{align} \label{ing-3-4-2}
\Li{\Delta u_{i,n}}{\Omega_\vd} \le C' (\Ln{V_{i,n}}{\Omega} + \Ln{\Delta u_{i,n}}{\Omega_\vd}).  
\end{align}

In the sequel, we let $N=2$, $3$. 
Let fix $x^0 \in \Omega \setminus \Omega_\vd$ arbitrarily and $B_\rho$ denote the ball with center $x^0$ and radius $\rho$. 
Choose $R>0$ such that $\con{B}_R \subset \Omega$ and $\vz \in C^\infty_0(B_R)$, $\vz=1$ in $B_{2R/3}$, 
$0 \le \vz \le 1$ elsewhere. 
For any $x \in B_{R/2}$, we can write 
\begin{align*}
(J_\ve U_{i,n})(x) = \int_{B_R} W(x-y) \Delta^2 (\vz J_\ve U_{i,n})(y) \, dy, 
\end{align*}
where $U_{i,n}$ is the function defined by \eqref{def-U} and $W$ is the fundamental solution of $\Delta^2$: 
\begin{align*}
W(x) = 
\begin{cases}
\gm_N |x|^2 (\log{|x|} -1) \quad & \text{if} \quad N=2, \\
-\gm_N \av{x} \quad & \text{if} \quad N=3, \\
\end{cases}
\end{align*}
where $\gm_N$ are constants chosen such that 
\begin{align*}
\Delta^2 W = \vd, 
\end{align*}
where $\vd$ denotes the Dirac measure (e.g., see \cite{F1}). 
Expanding $\Delta^2(\vz J_\ve U_{i,n})$ and performing integrations by parts, we obtain 
\begin{align} \label{u-W-1}
& (J_\ve U_{i,n})(x) \\
& \quad = \int_{B_{2R/3}}\!\!\!\!\!\!\!\! W(x-y) \Delta^2 (\vz J_\ve U_{i,n})(y) \, dy 
       + \int_{D_{R/3}} \!\!\! W(x-y) \Delta^2 (\vz J_\ve U_{i,n})(y) \, dy \notag \\
& \quad = \int_{B_{2R/3}} \!\!\!\!\!\!\!\! W(x-y) \vz(y) \Delta^2 (J_\ve U_{i,n})(y) \, dy \notag \\
& \qquad + \int_{D_{R/3}} \!\!\! W(x-y) \!\! \Bigm[ \!\! \Delta^2 \vz (J_\ve U_{i,n}) + 4 \nabla (\Delta \vz) \cdot \nabla (J_\ve U_{i,n}) 
          + 6 \Delta \vz \Delta(J_\ve U_{i,n}) \notag \\
& \qquad \qquad \qquad \qquad \qquad \quad
      +4 \nabla \vz \cdot \nabla \Delta (J_\ve U_{i,n}) + \vz \Delta^2(J_\ve U_{i,n}) \Bigm] (y) \, dy \notag \\
& \quad = \int_{B_R} \!\!\! W(x-y) \vz(y) \Delta^2 (J_\ve U_{i,n})(y) \, dy + \va_\ve(x), \notag
\end{align}
where $D_{R/3} := R_R \setminus B_{2R/3}$ and 
\begin{align*}
\va_\ve(x)&= \int_{D_{R/3}} W(x-y) \Bigm[ \Delta^2 \vz (J_\ve U_{i,n}) + 4 \nabla (\Delta \vz) \cdot \nabla (J_\ve U_{i,n}) 
          + 2 \Delta \vz \Delta(J_\ve U_{i,n}) \Bigm](y) \, dy \\
& \qquad \quad -4 \int_{D_{R/3}} \nabla W(x-y) \cdot \nabla \vz(y) \Delta(J_\ve U_{i,n})(y) \, dy. 
\end{align*}
Since it follows from a direct calculation that 
\begin{align*}
\left( \p{}{x_j}{2} - \dfrac{1}{2} \Delta \right)W(x) 
 = 
\begin{cases}
\gm_N \left( 2 x_j^2 |x|^{-2} -1 \right) \quad & \text{if} \quad N=2, \\
\gm_N x_j^2 |x|^{-3} \quad & \text{if} \quad N=3, 
\end{cases}
\end{align*}
one can verify that 
\begin{align} \label{W-bdd-b}
\left( \p{}{x_j}{2} - \dfrac{1}{2} \Delta \right)W \ge -c, 
\end{align}
where $c$ is a positive constant. 
Applying $\pd^2/\pd x_j^2 - \Delta/2$ to the both sides of \eqref{u-W-1} and using \eqref{W-bdd-b} and 
the fact that $\vz \Delta^2 (J_\ve U_{i,n}) \ge 0$, we obtain, if $x \in B_{R/2}$, 
\begin{align*}
\left( \p{}{x_j}{2} -\dfrac{1}{2} \Delta \right) J_\ve U_{i,n}(x) 
 \ge -c \int_{B_R} \vz(y) \Delta^2 (J_\ve U_{i,n})(y) \, dy + \left( \p{}{x_j}{2} -\dfrac{1}{2} \Delta \right) \va_\ve(x). 
\end{align*}
Since the integral in the right-hand side can be written as 
\begin{align*}
\int_{B_R} (J_\ve \vz)(y) \Delta^2 U_{i,n}(y) \, dy + \vb_\ve, 
\end{align*}
where $\vb_\ve \to 0$ as $\ve \downarrow 0$, we conclude that 
\begin{align} \label{2D-Lbd-ve}
\p{J_\ve U_{i,n}}{x_j}{2}(x) 
 & \ge -\dfrac{1}{2} \Li{\Delta J_\ve U_{i,n}}{B_{R/3}} -c \int_{B_R} (J_\ve \vz)(y) \Delta^2 U_{i,n}(y) \, dy \\
 & \qquad -c \vb_\ve + \left( \p{}{x_j}{2} -\dfrac{1}{2} \Delta \right) \va_\ve(x)
 \quad \text{in} \quad B_{R/3}. \notag
\end{align}
On the other hand, it also holds that 
\begin{align} \label{2D-Ubd-ve}
\p{J_\ve U_{i,n}}{x_j}{2} 
& = \Delta(J_\ve U_{i,n}) - \sum_{k \neq j} \p{J_\ve U_{i,n}}{x_k}{2} \\
& \le  \dfrac{N+1}{2} \Li{\Delta(J_\ve U_{i,n})}{B_{R/3}} + c(N-1) \int_{B_R} (J_\ve \vz)(y) \Delta^2 U_{i,n}(y) \, dy \notag \\
& \qquad + c (N-1) \vb_\ve - (N-1) \left( \p{}{x_j}{2} -\dfrac{1}{2} \Delta \right) \va_\ve(x)
   \quad \text{in} \quad B_{R/3}. \notag
\end{align}
Lemma \ref{C-F-lem-3} implies that 
\begin{align} \label{part-1}
\Li{\Delta(J_\ve U_{i,n})}{B_{R/3}} 
 &\le \Li{\Delta U_{i,n}}{B_{R/3}} \\
 \le C_1 E(u_0)^{\frac{1}{2}} & + (C_2+1) \Ln{V_{i,n}}{\Omega} 
      + C_3 \mu_{i,n}(D_{R/2}) + \Li{\Delta f}{B_{R/2}}. \notag 
\end{align}
Letting $\ve \downarrow 0$, we find 
\begin{align} \label{part-2}
\int_{B_R} (J_\ve \vz)(y) \Delta^2 U_{i,n}(y) \, dy 
 \to \int_{B_R} \vz(y) \Delta^2 U_{i,n}(y) \, dy \le \mu_{i,n}(B_R).  
\end{align}
Furthermore it follows from the Gagliardo-Nirenberg type interpolation inequality that 
\begin{align*}
\Li{\va_\ve}{B_{R/3}} 
 &\le C \{ \Ln{J_\ve U_{i,n}}{\Omega} + \Ln{\nabla (J_\ve U_{i,n})}{\Omega} + \Ln{\Delta (J_\ve U_{i,n})}{\Omega} \} \\
 &\le C \{ \Ln{J_\ve U_{i,n}}{\Omega} + \Ln{\Delta (J_\ve U_{i,n})}{\Omega} \} \\
 &\le C \{ \Ln{U_{i,n}}{\Omega} + \Ln{\Delta U_{i,n} }{\Omega} \}. 
\end{align*}
Observing 
\begin{align*}
\Ln{U_{i,n}}{\Omega} 
 \le \Ln{u_{i,n}}{\Omega} + \Ln{(\Delta^2)^{-1} V_{i,n}}{\Omega} 
 \le \Hnd{u_{i,n}} + C \Ln{V_{i,n}}{\Omega}, 
\end{align*}
we obtain 
\begin{align} \label{part-3}
\Li{\va_\ve}{B_{R/3}} 
 \le C_1' E(u_0)^{\frac{1}{2}} & + C_2' \Ln{V_{i,n}}{\Omega} 
      + C_3' \mu_{i,n}(D_{R/2}) + C_4 \Li{\Delta f}{B_{R/2}}
\end{align}
Recalling $\vb_\ve \to 0$ as $\ve \downarrow 0$ and letting $\ve \downarrow 0$ in \eqref{2D-Lbd-ve} and \eqref{2D-Ubd-ve}, 
we deduce from \eqref{part-1}--\eqref{part-3} that 
\begin{align} \label{pre-D2-bdd}
\Li{\p{u_{i,n}}{x_j}{2}}{B_{R/3}} 
 \le C_5 E(u_0)^{\frac{1}{2}} & + C_6 \Ln{V_{i,n}}{\Omega} 
      + C_7 \mu_{i,n}(B_R) + C_8 \Li{\Delta f}{B_{R/2}}
\end{align}
Since $x_j$ can be in any direction, the inequality \eqref{pre-D2-bdd} implies that 
\begin{align} \label{D2-bdd}
\Li{D^2 u_{i,n}}{B_{R/3}} 
\le C_5' E(u_0)^{\frac{1}{2}} + C_6' \Ln{V_{i,n}}{\Omega} + C_7' \mu_{i,n}(B_R) + C_8' \Li{\Delta f}{B_{R/2}}, 
\end{align} 
where the constants $C_5'$, $C_6'$, $C_7'$, and $C_8'$ are independent of $i$ and $n$. 
Recalling \eqref{ing-3-4-2}, along the same line as above, one can verify that 
\begin{align} \label{ing-3-4-3}
\Li{D^2 u_{i,n}}{\Omega_\vd} \le C (\Ln{V_{i,n}}{\Omega_\vd} + E(u_0)^{\frac{1}{2}}), 
\end{align}
where the constant $C$ depends only on $\Omega_\vd$. 
Since $\Omega \setminus \Omega_\vd$ is compact, combining \eqref{D2-bdd} with \eqref{ing-3-4-3}, 
we obtain the assertion $u_{i,n} \in W^{2,\infty}(\Omega)$ and 
\begin{align} \label{ing-3-4-4}
\Li{D^2 u_{i,n}}{\Omega}
\le C E(u_0)^{\frac{1}{2}} + C \Ln{V_{i,n}}{\Omega} + C \mu_{i,n}(\Omega) + C \Li{\Delta f}{\Omega}
\end{align}
Finally multiplying \eqref{ing-3-4-4} by $\tau_n$ and summing over $i= 1, \cdots, n$, 
we conclude from \eqref{est-u_n-t} and \eqref{me-bd} that 
\begin{align*}
&\tau_n \sum^n_{i=1} \Li{D^2 u_{i,n}}{\Omega}^2  \\
& \quad \le C T E(u_0) + C \int^T_0 \Ln{V_n(t)}{\Omega}^2 \, dt 
      + C \tau_n \sum^n_{i=1} \mu_{i,n}(\Omega)^2 + C T \Li{\Delta f}{\Omega}^2 \\
& \quad \le C T E(u_0) + 2 C E(u_0) + C + C T \Li{\Delta f}{\Omega}^2. 
\end{align*} 
This completes the proof. 
\end{proof}

When we restrict to dimensions $N \le 3$, Proposition \ref{compactness} implies that $u_{i,n}$ is continuous. 
Under such restriction, we define 
\begin{align}
\mC_{i,n}:= \{ x \in \Omega :\  u_{i,n}(x) = f(x) \}, \label{coincide} \\
\mN_{i,n} := \{ x \in \Omega :\  u_{i,n}(x) > f(x) \}. \label{not-coincide}
\end{align}
It is clear that $\mC_{i,n} \cup \mN_{i,n} = \Omega$. 
We can show a relation between the support of $\mu_{i,n}$ and the sets. 

\begin{lem} \label{measure-coincide}
Let $N \le 3$. If $x_0 \in \mN_{i,n}$, then there exists a neighborhood 
of $x_0$ such that $\mu_{i,n}(\mN_{i,n})=0$. Furthermore we have 
\begin{align} \label{supp-mu-in}
{\rm supp}\, \mu_{i,n} \subseteq \mC_{i,n}. 
\end{align}
\end{lem}
\begin{proof}
Let $N \le 3$ and fix $x^0 \in \mN_{i,n}$ arbitrarily. 
Since $\mN_{i,n}$ is an open set, there exist a constant $\vd>0$ and a neighborhood $W$ of $x^0$ such that 
\begin{align*}
u_{i,n}(x) - f(x) > \vd \quad \text{for all} \quad x \in W. 
\end{align*}
Notice that $u_{i,n}$ satisfies 
\begin{align} \label{vari-ineq-1}
\int_\Omega \Delta u_{i,n} \Delta(u_{i,n} - \vp) \, dx \le - \int_\Omega V_{i,n} (u_{i,n} - \vp) \, dx 
\end{align}
for any $\vp \in K$, for $u_{i,n}$ is a solution of \eqref{incre-mini}. 
Then for any $\vz \in C^\infty_0(W)$ with $0 \le \vz \le \vd /2$, the function 
\begin{align*}
\vq = u_{i,n} - \vz
\end{align*}
belongs to $K$. Taking this $\vq$ as $\vp$ in \eqref{vari-ineq-1}, we have 
\begin{align*}
\int_\Omega \left[ \Delta u_{i,n} \Delta \vz + V_{i,n} \vz \right] dx \le 0,  
\end{align*}
Since $\mu_{i,n} \ge 0$, this asserts that 
\begin{align*}
\int_\Omega \left[ \Delta u_{i,n} \Delta \vz + V_{i,n} \vz \right] dx =0, 
\end{align*}
i.e.,  $\mu_{i,n}=0$ in $W$. 
\end{proof}

\smallskip


\section{Existence and regularity of solutions to problem \eqref{P}}
We first prove a convergence result which holds in any dimension $N \ge 1$.  
\begin{thm} \label{convergence-1}
Let $u_n$ be the piecewise linear interpolation of $\{ u_{i,n} \}$.  
Then there exists a function $$u \in L^\infty([0,+\infty);H^{2}_{0}(\Omega)) \cap H^1_{loc}(0,+\infty;L^2(\Omega))$$ such that
\begin{align}
u_n \rightharpoonup u \quad\text{in} \quad L^2(0,T;H^2_0(\Omega)) \cap H^1(0,T;L^2(\Omega)) 
\quad \text{as} \quad n \to +\infty\,, 
\end{align}
up to a subsequence, for any $0< T < +\infty$. 
Moreover 
\begin{align*}
\int_0^T\int_\Omega u_t^2\,dx\,dt\le 2E(u_0)\,,
\end{align*}
$u(x,t) \ge f(x)$ for a.e. $x\in\Omega$ and for every $t \in [0,+\infty)$, and for each $\va \in (0, \tfrac{1}{2})$ it holds
\begin{align} \label{holder-conv}
u_n \to u \quad \text{in} \quad C^{0, \va}([0,T]; L^2(\Omega)) \quad \text{as} \quad n \to +\infty\,.
\end{align}
\end{thm}
\begin{proof}
Recalling that $u_n(x,\cdot)$ is absolutely continuous on $[0,T]$, for all $t_1$, $t_2 \in [0,T]$ with $t_1 < t_2$, 
H\"older's inequality and Fubini's Theorem give us 
\begin{align*}
\Lp{u_n(\cdot, t_2) - u_n(\cdot, t_1)}{2}{\Omega} 
 &= \left( \int^L_0 \left( \int^{t_2}_{t_1} \p{u_n}{t}{}(x,t) \, dt \right)^2 dx \right)^{\frac{1}{2}} \\
 &\le \left( \int^{t_2}_{t_1} \Lp{\p{u_n}{t}{}(\cdot,t)}{2}{\Omega}^2 dt  \right)^{\frac{1}{2}} (t_2 - t_1)^{\frac{1}{2}}. 
\end{align*} 
Then it follows from \eqref{est-u_n-t} that 
\begin{align} \label{u_n-equi-l2}
\int_{t_1}^{t_2}\int_\Omega u_t^2 \,dx\,dt\le 2E(u_0)
\end{align}
and
\begin{align} \label{u_n-equi-conti}
\Lp{u_n(\cdot,t_2) - u_n(\cdot,t_1)}{2}{\Omega} \le \sqrt{2E(u_0)} (t_2 - t_1)^{\frac{1}{2}}. 
\end{align}
Since \eqref{u-in-uni-bd} yields that 
\begin{align} \label{u_n-ub-D2}
\sup_{t \in [0,T]}{\Ln{\Delta u_n(\cdot,t)}{\Omega}} 
 \le \sup_{1 \le i \le n}{\Ln{\Delta u_{i,n}}{\Omega}} \le \sqrt{2 E(u_0)}, 
\end{align}
there exists a function $u \in L^2(0,T;H^2_0(\Omega))$ such that $u_n \rightharpoonup u$ 
in $L^2(0,T;H^2_0(\Omega))$ up to a subsequence. 
On the other hand, the estimate \eqref{est-u_n-t} implies that 
\begin{align} \label{V-conv}
V_n= \p{u_n}{t}{} \rightharpoonup \p{u}{t}{} \quad \text{in} \quad L^2(0,T;L^2(\Omega)). 
\end{align}
This means that $\pd u/\pd t \in L^2(0,T;L^2(\Omega))$, i.e., $u \in H^1(0,T;L^2(\Omega))$. 
Combining \eqref{u_n-equi-conti} with Ascoli-Arzel\`a's Theorem (see e.g. \cite[Proposition 3.3.1]{AGS}), 
we conclude \eqref{holder-conv}. 

Since \eqref{u_n-ub-D2} means that $\{ u_n(t) \}$ is uniformly bounded in $H^{2}_{0}(\Omega)$ 
with respect to $t \in [0,T]$ and $n \in \N$, we deduce from \eqref{holder-conv} that, for each $t \in [0,T]$ 
\begin{align} \label{u_n(t)-D2-w}
u_n(t) \rightharpoonup u(t) \quad \text{in} \quad H^2_0(\Omega) 
\end{align}
up to a subsequence. This asserts that $u \in L^\infty([0,T];H^{2}_{0}(\Omega))$. 
Moreover, Proposition \ref{compactness} implies that for each $t \in [0,T]$ 
\begin{align} \label{u_n-Lq-s}
u_n(t) \to u(t) \quad \text{in} 
\begin{cases}
C^{1,\gm}(\Omega) \quad  \text{for} \quad 0 < \gm < \frac{1}{2} \quad & \text{if} \quad N=1, \\
C^{0,\gm}(\Omega) \quad  \text{for} \quad 0 < \gm < 2-\frac{N}{2} & \text{if} \quad N=2, 3, \\
L^q(\Omega) \quad \,\,\,\,  \text{for} \quad 0 < q < +\infty \quad & \text{if} \quad N=4, \\  
L^q(\Omega) \quad \,\,\,\, \text{for} \quad 0 < q < \frac{2N}{N-4} \quad & \text{if} \quad N \ge 5.
\end{cases}
\end{align}
In particular, if $N \ge 4$,  
\begin{align} \label{u_n-m-c}
u_n(t) \to u(t) \quad \text{a.e. in} \quad \Omega 
\end{align}
up to a subsequence. 
Since $u_n(t) \ge f$ a.e. in $\Omega$ for each $n \in \N$ and $t \in [0,T]$, 
the fact \eqref{u_n-Lq-s}-\eqref{u_n-m-c} yields that $u(t) \ge f$ a.e. in $\Omega$ for each $t \in [0,T]$. 
This completes the proof. 
\end{proof}

When $N=1$, we can improve the convergence result obtained in Theorem \ref{convergence-1}: 
\begin{thm} \label{conv-D1}
Let $N=1$. 
Let $u$ be the function obtained by Theorem \ref{convergence-1}. 
Then it holds that $u \in L^2(0,T;W^{2, \infty}(\Omega)) \cap C^{0, \vb}([0,T]; C^{1,\va}(\Omega))$ and 
\begin{align} 
& u_n \to u \quad \text{weakly{\rm *} in} \quad L^2(0,T;W^{2, \infty}(\Omega))
 \quad \text{as} \quad n \to \infty, \label{L2W2i-N1} \\
& u_n \to u \quad \text{in} \quad C^{0, \vb}([0,T];C^{1,\va}(\Omega)) 
 \quad \text{as} \quad n \to \infty \label{u_n-conv-holder-vb}
\end{align}
for every $\va \in (0, \tfrac{1}{2})$ and $\vb \in (0, \tfrac{1-2\va}{8})$. 
Furthermore $u(\cdot,t) \to u_0$ in $C^{1, \va}(\Omega)$ as $t \downarrow 0$. 
\end{thm}
\begin{proof}
Fix $T>0$ and $n \in \N$. 
To begin with, we shall prove \eqref{L2W2i-N1}. By \eqref{est-W2infinity} we see that 
$u_n$ is uniformly bounded in $L^2(0,T;W^{2,\infty}(\Omega))$ with respect to $n \in \N$. 
Since $L^2(0,T;W^{2,\infty}(\Omega))$ is the dual of $L^2(0,T;W^{2,1}(\Omega))$, 
Banach-Alaoglu's Theorem asserts that $u_n$ subconverges to $u$ weakly* in $L^2(0,T;W^{2,\infty}(\Omega))$. 
In particular, combining \eqref{est-W2infinity} with   
\begin{align*}
\| u \|_{L^2(0,T;W^{2,\infty}(\Omega))} \le \liminf_{n \to +\infty} \| u_n \|_{L^2([0,T];W^{2,\infty}(\Omega))}, 
\end{align*}
we observe that $u \in L^2(0,T;W^{2,\infty}(\Omega))$. 

Next we prove \eqref{est-W2infinity}. In the sequel we let $\Omega= (0,L)$. 
Let us define the function $g:= u_n(\cdot,t_2)- u_n(\cdot, t_1)$. 
Since $g \in H^2_0(\Omega)$ for each $t_1$, $t_2 \in [0, T]$ with $t_1 < t_2$, we have 
\begin{align} \label{pre-inter-1}
\int_\Omega (g'(x))^2 \, dx 
 = - \int_\Omega g(x) g''(x) \, dx 
 \le \Lp{g}{2}{\Omega} \Lp{g''}{2}{\Omega}, 
\end{align}
and 
\begin{align} \label{pre-inter-2}
(g'(x))^2 = \int^x_0 \{ (g'(x))^2 \}' \, dx \le 2 \Lp{g'}{2}{\Omega} \Lp{g''}{2}{\Omega}. 
\end{align}
Then \eqref{pre-inter-1} and \eqref{pre-inter-2} yield 
\begin{align}
\Li{g'}{\Omega} \le \sqrt{2} \Lp{g''}{2}{\Omega}^{\frac{3}{4}} \Lp{g}{2}{\Omega}^{\frac{1}{4}}. 
\end{align}
Since $\Lp{g''}{2}{\Omega} \le 2 \sup_{i,n}{\Lp{u''_{i,n}}{2}{\Omega}}$, we observe from \eqref{u_n-ub-D2} that  
\begin{align*}
\Li{g'}{\Omega} \le \sqrt{2} ( 2 \sqrt{2 E(u_0)})^{\frac{3}{4}} \Lp{g}{2}{\Omega}^{\frac{1}{4}}. 
\end{align*}
Then, by \eqref{u_n-equi-conti}, we obtain 
\begin{align} \label{u_n'-equi-conti}
\Li{\p{u_n}{x}{}(\cdot,t_2) - \p{u_n}{x}{}(\cdot,t_1) }{\Omega}
 \le 2^{\frac{13}{8}} \sqrt{E(u_0)} (t_2 - t_2)^{\frac{1}{8}}.
\end{align}
Moreover, by the Mean Value Theorem, there exists $\bar{x} \in \Omega$ such that 
\begin{align*}
g(\bar{x})= \dfrac{1}{L} \int^L_0 g(x) \, dx, 
\end{align*}
and then 
\begin{align*}
\av{g(x)} \le \av{g(x) - g(\bar{x})} + \av{g(\bar{x})} 
 \le L \Li{g'}{\Omega} + \dfrac{1}{\sqrt{L}} \Lp{g}{2}{\Omega}
\end{align*}
for each $x \in [0,L]$. 
Thus, by \eqref{u_n-equi-conti} and \eqref{u_n'-equi-conti}, we find 
\begin{align} \label{u_n-i-equi-conti}
\Li{u_n(\cdot,t_2) - u_n(\cdot,t_1)}{\Omega} 
 &\le 2^{\frac{13}{8}} L \sqrt{E(u_0)}  (t_2 - t_1)^{\frac{1}{8}} + \sqrt{\dfrac{E(u_0)}{L}} (t_2 - t_1)^{\frac{1}{2}} \\
 &\le 4 L \sqrt{E(u_0)} \left( 1 + \dfrac{T^{\frac{3}{8}}}{4 \sqrt{L}} \right) (t_2 - t_1)^{\frac{1}{8}}. \notag
\end{align}
Furthermore, for each $\va \in (0, \tfrac{1}{2})$,  we have 
\begin{align} \label{g'-holder}
\av{g'}_\va := \sup{\left\{ \dfrac{\av{g'(x)-g'(y)}}{\av{x-y}^\va} \Bigm| x, y \in \Omega, x \neq y \right\}} 
 \le \av{g'}^{2 \va}_{\frac{1}{2}} (2 \Li{g'}{\Omega})^{1-2\va}. 
\end{align}
Using Morrey's inequality, it is followed from \eqref{u_n-ub-D2} that 
\begin{align*}
\av{\p{u_n}{x}{}(\cdot,t_2) - \p{u_n}{x}{}(\cdot,t_1) }_{\frac{1}{2}} 
 \le K_M \Hn{\p{u_n}{x}{}(\cdot,t_2) - \p{u_n}{x}{}(\cdot,t_1)}{1} 
 \le 2 K_M C_0 \sqrt{E(u_0)},  
\end{align*}
where $K_M$ denotes the constant of Morrey's inequality. 
Then, from \eqref{u_n'-equi-conti} and \eqref{g'-holder}, we deduce that 
\begin{align} \label{u_n'-holder-equi-conti}
\av{\p{u_n}{x}{}(\cdot,t_2) - \p{u_n}{x}{}(\cdot,t_1)}_\va 
 \le 2 \sqrt{E(u_0)} (K_M C_0)^{2\va} \left( 1 + \dfrac{T^{\frac{3}{8}}}{4 \sqrt{L}} \right)^{1-2\va} (t_2 - t_1)^{\frac{1-2\va}{8}}. 
\end{align}

Therefore it follows from \eqref{u_n'-equi-conti}, \eqref{u_n-i-equi-conti}, and \eqref{u_n'-holder-equi-conti}, that 
for every $\va \in (0, \tfrac{1}{2})$, $u_n$ is uniformly equicontinuous with respect to the $C^{1, \va}(\Omega)$-norm topology 
and that 
\begin{align} \label{equi-conti-D1}
\Cn{u_n(\cdot,t_2) - u_n(\cdot,t_1)}{1, \va}{\Omega} 
 \le C (t_2 - t_1)^{\frac{1-2\va}{8}} 
\end{align}
for some $C(L, E(u_0), \va, T)>0$. 
We then obtain \eqref{u_n-conv-holder-vb} by applying the Ascoli-Arzel\`a's Theorem (see e.g. \cite[Proposition 3.3.1]{AGS}). 
Finally, since 
\begin{align*}
\Cn{u_n(\cdot,t) - u_n(\cdot,t_1)}{1, \va}{\Omega} \to 0 \quad \text{as} \quad t \to t_1, 
\end{align*}
we obtain the conclusion by selecting $t_1=0$. 
\end{proof}

When $N=2$, $3$, we can also improve the result obtained in Theorem \ref{convergence-1}: 
\begin{thm} \label{conv-D23}
Let $N=2$, $3$. 
Let $u$ be the function obtained by Theorem \ref{convergence-1}. 
Then it holds that $u \in L^2(0,T; W^{2,\infty}(\Omega)) \cap C^{0, \vb}([0,T] ; C^{0,\gm}(\Omega))$ and 
\begin{align} 
& u_n \to u \quad \text{weakly{\rm *} in} \quad L^2(0,T;W^{2,\infty}(\Omega)) \quad \text{as} \quad n \to + \infty, 
   \label{L2W2i} \\
& u_n \to u \quad \text{in} \quad C^{0, \vb}([0,T]; C^{0,\gm}(\Omega)) \quad \text{as} \quad n \to +\infty 
\label{u_n-conv-holder-D23}
\end{align}
for every 
\begin{align*}
0 < \vb < \left(\frac{1}{2} - \frac{N}{8} \right) \left( 1 - \frac{\gm}{2-N/2} \right), \quad 0 < \gm < 2 - \dfrac{N}{2}. 
\end{align*}
Furthermore $u(\cdot,t) \to u_0$ in $C^{0, \gm}(\Omega)$ as $t \downarrow 0$. 
\end{thm}
\begin{proof}
Let $N=2$, $3$. Fix $T>0$ and $n \in \N$. 
To begin with, the convergence \eqref{L2W2i} follows from the same line as in the proof of \eqref{L2W2i-N1}. 
In the sequel, we shall prove \eqref{u_n-conv-holder-D23}. 
For each $t_1$, $t_2 \in [0,T]$ with $t_1 < t_2$, set  
\begin{align*}
g(x) := u_n(x,t_2) - u_n(x,t_1). 
\end{align*}
By \eqref{u_n-equi-conti}, we have already known 
\begin{align} \label{g-L2-bd}
\Ln{g}{\Omega} \le (2 E(u_0))^{\frac{1}{2}} (t_2 - t_1)^{\frac{1}{2}}. 
\end{align}
Since \eqref{u-in-uni-bd} asserts that  
\begin{align*}
\Hnd{g} \le 2 (2 E(u_0))^{\frac{1}{2}}, 
\end{align*}
combining this with \eqref{g-L2-bd} and the interpolation inequality 
\begin{align} \label{inter-GN}
\Li{g}{\Omega} \le C \Ln{g}{\Omega}^{1-\frac{N}{4}} \Hnd{g}^{\frac{N}{4}}, 
\end{align}
we obtain 
\begin{align} \label{pre-equi-conti-g-D23}
\Li{g}{\Omega} \le C \Ln{g}{\Omega}^{1-\frac{N}{4}} \le C (t_2 -t_1)^{\frac{1}{2} - \frac{N}{8}},  
\end{align}
where the constant $C$ is independent of $n$. 
For each $\gm \in (0, 2-N/2)$, we obtain 
\begin{align*}
| g |_\gm := \sup\left\{ \dfrac{\av{g(x)-g(y)}}{\av{x-y}^\gm} \biggm| x, y \in \Omega, x \neq y \right\} 
               \le | g |^{\gm/(2-N/2)}_{2-N/2} (2 \Li{g}{\Omega})^{1- \frac{\gm}{2-N/2}}. 
\end{align*}
Since it follows from Sobolev's embedding theorem that 
\begin{align*}
\Cn{g}{0, 2-N/2}{\Omega} \le C \Hnd{g} \le C E(u_0)^{\frac{1}{2}}, 
\end{align*}
we get 
\begin{align} \label{pre-equi-conti-g-D23-2}
| g |_\gm \le C (t_2 - t_1)^{\left(\frac{1}{2} - \frac{N}{8} \right) \left( 1 - \frac{\gm}{2-N/2} \right)}
\end{align}
Therefore we deduce from \eqref{pre-equi-conti-g-D23} and \eqref{pre-equi-conti-g-D23-2} that 
$u_n$ is uniformly equicontinuous with respect to the $C^{0,\gm}$-norm topology for each $\gm \in (0, 2-N/2)$, and that 
\begin{align} \label{equi-conti-D23}
\Cn{u_n(\cdot,t_2) - u_n(\cdot, t_1)}{0, \gm}{\Omega} \le C (t_2 - t_1)^{\left(\frac{1}{2} - \frac{N}{8} \right) \left( 1 - \frac{\gm}{2-N/2} \right)} 
\end{align}
for some constant $C=C(\Omega, E(u_0), \gm, T)>0$. 
By the Ascoli-Arzel\`a's Theorem (see e.g. \cite[Proposition 3.3.1]{AGS}), we get \eqref{u_n-conv-holder-D23}.  
Finally, since 
\begin{align*}
\Cn{u_n(\cdot,t) - u_n(\cdot,t_1)}{0, \gm}{\Omega} \to 0 \quad \text{as} \quad t \to t_1, 
\end{align*}
we obtain the conclusion by selecting $t_1=0$. 
\end{proof}

Regarding the piecewise constant interpolation $\tilde{u}_n$ for $\{ u_{i,n} \}$ defined in Definition \ref{def-piece-inter}, 
we can verify the following: 
\begin{lem} \label{piece-inter}
Let $\tilde{u}_n$ be the piecewise constant interpolation of $\{ u_{i,n} \}$. 
If $N=1$, then 
\begin{align} \label{tilde-u-conv-D1}
\tilde{u}_n \to u \quad \text{in} \quad L^{\infty}([0,T]; C^{1,\gm}(\Omega)) \quad \text{as} \quad n \to +\infty 
\end{align}
for every $\gm \in (0, 1/2)$, where $u$ is the function obtained in Theorem \ref{convergence-1}. 
If $N=2$, $3$, then 
\begin{align} \label{tilde-u-conv-D23}
\tilde{u}_n \to u \quad \text{in} \quad L^\infty([0,T]; C^{0,\gm}(\Omega)) \quad \text{as} \quad n \to +\infty 
\end{align}
for every $\gm \in (0, 2-N/2)$. Furthermore, for any $N \ge 1$, it holds that 
\begin{align} \label{tilde-Delta-u-conv}
\Delta \tilde{u}_n \rightharpoonup \Delta u \quad \text{in} \quad L^2(0,T; L^2(\Omega)) \quad \text{as} \quad n \to +\infty. 
\end{align}
\end{lem}
\begin{proof}
By \eqref{u-in-uni-bd} we see that $\tilde{u}_n \in L^\infty([0,T];H^2_0(\Omega))$. 
Since $N \le 3$, Proposition \ref{compactness} implies that 
\begin{align*}
\tilde{u}_n \in 
\begin{cases}
L^\infty([0,T]; C^{1,\gm}(\Omega)) \quad \text{for} \quad 0 < \gm < \frac{1}{2} \quad & \text{if} \quad N=1, \\
L^\infty([0,T]; C^{0,\gm}(\Omega)) \quad \text{for} \quad 0 < \gm < 2 - \frac{N}{2} \quad & \text{if} \quad N=2, 3. \\
\end{cases}
\end{align*}
Then, along the same line as in the proof of Theorem \ref{convergence-1}, we verify that $\tilde{u}_n(t)$ converges 
to a function $\tilde{u}(t)$, with $\tilde{u}(x,t) \ge f(x)$ in $\Omega$, for each $t \in [0,T]$ 
in $C^{1, \gm}(\Omega)$ if $N=1$ and $C^{0,\gm}(\Omega)$ if $N=2$, $3$. 

We shall show that $\tilde{u}$ coincides with $u$ which is obtained as the limit of $u_n$ . 
Let us fix $t \in [0,T]$ arbitrarily. 
Then there exists a sequence of intervals $\{ [(i_n-1) \tau_n, i_n \tau_n) \}_{n \in \N}$ such that 
$t \in [(i_n-1) \tau_n, i_n \tau_n)$ for each $n \in \N$. 
Recalling Definitions \ref{def-linear-inter}-\ref{def-piece-inter}, 
if $N=1$, we observe from \eqref{equi-conti-D1} that  
\begin{align*}
\Cn{\tilde{u}_n(t) - u_n(t)}{1,\gm}{\Omega} 
& = \Cn{u_{i,n} - u_n(t)}{1,\gm}{\Omega} \\
& = \Cn{u_n(i_n \tau_n) - u_n(t)}{1,\gm}{\Omega} \\
& \le C (i_n \tau_n - t)^{\frac{1-2\gm}{8}} 
   \le C \tau_n^{\frac{1-2\gm}{8}} \to 0 \quad \text{as} \quad n \to + \infty, 
\end{align*}
and if $N=2$, $3$, we deduce from \eqref{equi-conti-D23} that 
\begin{align*}
\Cn{\tilde{u}_n(t) - u_n(t)}{0,\gm}{\Omega} 
& = \Cn{u_n(i_n \tau_n) - u_n(t)}{0,\gm}{\Omega} \\
& \le C \tau_n^{\left(\frac{1}{2} - \frac{N}{8} \right) \left( 1 - \frac{\gm}{2-N/2} \right)} 
\to 0 \quad \text{as} \quad n \to + \infty.  
\end{align*}
Hence we obtain \eqref{tilde-u-conv-D1} and \eqref{tilde-u-conv-D23}. 

Finally we prove \eqref{tilde-Delta-u-conv}. 
It follows from Definitions \ref{def-linear-inter} and \ref{def-piece-inter} that 
\begin{align*}
u_n(x,t) - \tilde{u}_n(x,t) = \dfrac{1}{\tau_n}(t- i \tau_n) (u_{i,n}(x) - u_{i-1,n}(x)), 
\end{align*}
so that, 
\begin{align} \label{u_n-to-tu_n-L2}
&\dfrac{1}{2} \sup_{t \in [0,T]} \int_\Omega \av{u_n(x,t) - \tilde{u}_n(x,t) }^2 \, dx \\
&\quad \le \sum^{n}_{i=1} \sup_{t \in [(i-1) \tau_n, i \tau_n]} 
                 \dfrac{(t - i \tau_n)^2}{\tau_n} \int_\Omega \dfrac{1}{2 \tau_n} (u_{i,n}(x) - u_{i-1,n}(x))^2 \, dx \notag \\
&\quad \le  \tau_n \sum^{n}_{i=1} (E(u_{i-1,n}) - E(u_{i,n})) \notag \\ 
&\quad=  \tau_n (E(u_0) - E(u_{n,n})) 
 \le \tau_n E(u_0) \to 0 \quad \text{as} \quad n \to + \infty. \notag
\end{align}
Then we observe that for any $\vp \in C^\infty_c(\Omega)$  
\begin{align*}
\int^T_0 \!\!\! \int_\Omega (\Delta u_n - \Delta \tilde{u}_n) \vp \, dx dt 
 = \int^T_0 \!\!\! \int_\Omega (u_n - \tilde{u}_n) \Delta \vp \, dx dt 
 \to 0 \quad \text{as} \quad n \to \infty. 
\end{align*}
\end{proof}

Let us define $\mu_n$ as 
\begin{align} \label{def-mu-n}
\mu_n(t)= \mu_{i,n}  \quad \text{if} \quad t \in [(i-1) \tau_n, i \tau_n). 
\end{align}


\smallskip 

\noindent {\it Proof of Theorem \ref{main-thm}.}
Let $u$ be the function in Theorem \ref{convergence-1}. 
To begin with, we prove that $u$ is a weak solution of \eqref{P}. 
Since $u_{i,n}$ and $V_{i,n}$ satisfy 
\begin{align*}
\int_{\Omega} \left[ V_{i,n} (\vp - u_{i,n}) + \Delta u_{i,n} \Delta(\vp - u_{i,n}) \right] \, dx \ge 0 
\end{align*}
for any $\vp \in K$, we observe that 
\begin{align*} 
& \int^{T}_{0}\!\!\!\! \int_{\Omega} \left[ V_{n} (w - \tilde{u}_{n}) 
      + \Delta \tilde{u}_{n} \Delta (w-\tilde{u}_{n}) \right] \, dx dt \\
& \qquad = \sum^{n}_{i=1} \int^{i \tau_{n}}_{(i-1)\tau_{n}} \! \int_{\Omega} 
    \left[ V_{i,n} (w- u_{i,n}) 
      + \Delta u_{i,n} \Delta (w- u_{i,n}) \right] \, dx dt \ge 0,  
\end{align*}
i.e., 
\begin{align} \label{pre-weak}
\int^{T}_{0}\!\!\!\! \int_{\Omega} \left[ V_{n} w + \Delta \tilde{u}_{n} \Delta w \right] \, dx dt 
\ge \int^{T}_{0}\!\!\!\! \int_{\Omega} \left[ V_{n} \tilde{u}_{n} 
      + \av{\Delta \tilde{u}_{n}}^{2} \right] \, dx dt  \quad \text{for all} \quad w \in \mathcal{K}. 
\end{align} 
It follows from \eqref{V-conv} that 
\begin{align} \label{pre-weak-1}
\int^{T}_{0}\!\!\!\! \int_{\Omega} V_{n} w \, dxdt \to \int^{T}_{0}\!\!\!\! \int_{\Omega} u_{t} w \, dxdt 
\quad \text{as} \quad n \to +\infty. 
\end{align}
Moreover Lemma \ref{piece-inter} gives us that 
\begin{align}\label{pre-weak-2}
\int^{T}_{0}\!\!\!\! \int_{\Omega} \Delta \tilde{u}_{n} \Delta w \, dx dt \to 
\int^{T}_{0}\!\!\!\! \int_{\Omega} \Delta u \Delta w \, dx dt \quad 
\text{as} \quad n \to + \infty, 
\end{align}
and 
\begin{align} \label{pre-weak-3}
\liminf_{n \to +\infty}\int^{T}_{0}\!\!\!\! \int_{\Omega} \av{\Delta \tilde{u}_{n}}^{2} \, dx dt 
 \ge \int^{T}_{0}\!\!\!\! \int_{\Omega} \av{\Delta u}^{2} \, dx dt.  
\end{align}
Combining \eqref{holder-conv} with \eqref{u_n-to-tu_n-L2}, we have 
\begin{align} \label{tu_n-to-u-L2}
\tilde{u}_n \to u \quad \text{as} \quad n \to +\infty \quad \text{in} \quad L^2(0,T;L^2(\Omega)). 
\end{align}
Then \eqref{V-conv} and \eqref{tu_n-to-u-L2} imply that 
\begin{align} \label{pre-weak-4}
\int^{T}_{0}\!\!\!\! \int_{\Omega} V_{n} \tilde{u}_{n} \, dx dt  
 \to \int^{T}_{0}\!\!\!\! \int_{\Omega} u_{t} u \, dx dt \quad \text{as} \quad n \to +\infty,    
\end{align}
e.g., see \cite{Z}, Proposition 23.9. 
By virtue of \eqref{pre-weak}--\eqref{pre-weak-3} and \eqref{pre-weak-4}, we assert that  
\begin{align} \label{exist-weak}
\int^{T}_{0}\!\!\!\! \int_{\Omega} \left[ u_{t} (w - u) + \Delta u \Delta (w - u) \right] \, dx dt \ge 0 
\quad \text{for all} \quad w \in \mathcal{K},  
\end{align}
i.e., $u$ is a weak  solution of \eqref{P}. 

For any $\vp \in C^\infty_c(\Omega \times (0,T))$ with $\vp \ge 0$, we verify that $w := u + \vp \in \mathcal{K}$. 
Hence it follows from \eqref{exist-weak} that 
\begin{align} \label{pre-existence-2}
\int^T_0 \!\!\! \int_\Omega \left[ u_t(x,t) \vp(x,t) + \Delta u(x,t) \Delta \vp(x,t) \right] \, dx dt \ge 0.   
\end{align}
Since $\vp$ is arbitrary, \eqref{pre-existence-2} implies that 
\begin{align} \label{D^2u+u_t-positive}
u_{t}(x,t) + \Delta^2 u(x,t) \ge 0 \quad \text{a.e. in} \quad \Omega \times (0,T),  
\end{align}
where $\Delta^2 u$ is written in the sense of distribution. 
Moreover, the regularity of $u$ follows from Theorems \ref{convergence-1}--\ref{conv-D23}.

\smallskip

We now prove \eqref{bdd-measure-sense}.  
By \eqref{def-mu-n} and Theorem \ref{bilaplace-measure}, we observe that  
\begin{align} \label{mu_n-ubd}
\| \mu_n \|_{L^2([0,T];\mM(\Omega))} &:=\int^T_0 \!\!\! \left(\int_\Omega d \mu_n \right)^2 dt \\
& = \sum^n_{i=1} \int^{i \tau_n}_{(i-1) \tau_n} \! \left( \int_\Omega d\mu_{i,n} \right)^2 dt 
 = \tau_n \sum^n_{i=1} \mu_{i,n}(\Omega)^2 < C.  \notag
\end{align}
This implies that 
\begin{align*}
\mu_n \rightharpoonup \con{\mu} \quad \text{weakly in} \quad L^2(0,T; \mM(\Omega))  
\end{align*}
up to a subsequence. 
Setting 
\begin{align*}
\mu := u_{t} + \Delta^2 u,  
\end{align*}
we observe from \eqref{D^2u+u_t-positive} that $\mu$ is a measure on $\Omega\times (0,T)$, and 
there holds $\con{\mu}= \mu$ by uniqueness of the limit. 
Since $\mu_n$ converges to $\mu$ weakly in $L^2(0,T; \mM(\Omega))$, it follows from \eqref{mu_n-ubd} that 
\begin{align*}
\| \mu \|_{L^2(0,T; \mM(\Omega))} \le \liminf_{n \to \infty} \| \mu_n \|_{L^2(0,T; \mM(\Omega))} \le C. 
\end{align*}
This is equivalent to \eqref{bdd-measure-sense}, and implies that $\mu$ is a positive Radon measure 
on $\Omega$ for a.e. $t\in (0,T)$.

\smallskip

Finally, when $N \le 3$, we prove that $u$ satisfies the problem \eqref{P} in the sense of distribution. 
To prove this assertion, it is sufficient to show that, if $u>f$, then $u_{t} + \Delta^2 u =0$ holds. 
Let us set 
\begin{align*}
\mN:= \{ (x,t) \in \Omega \times (0,T) :\  u(x,t) > f(x) \}. 
\end{align*}
Since $u$ is continuous in $\Omega \times (0,T)$ by Theorems \ref{conv-D1} and \ref{conv-D23}, 
$\mN$ is an open set, so that, for any $(x^0, t^0) \in \mN$, there exist 
$\vd>0$ and a neighborhood $W \times (t_1, t_2)$ of $(x^0, t^0)$ such that 
\begin{align} \label{u>f+vd}
u(x,t) - f(x) > \vd \quad \text{in} \quad W \times (t_1, t_2). 
\end{align}
Lemma \ref{piece-inter} implies that there exists a number $N>0$ such that 
\begin{align*}
\tilde{u}_n(x,t) > u(x,t) - \dfrac{\vd}{2} \quad \text{in} \quad W \times (t_1, t_2) \quad \text{for any} \quad n > N. 
\end{align*}
Combining this with \eqref{u>f+vd}, we have, for any $n>N$,  
\begin{align} \label{tilde-u>f+vd/2}
\tilde{u}_n(x,t) > f(x) + \dfrac{\vd}{2} \quad \text{in} \quad W \times (t_1, t_2). 
\end{align}
Let $\vz \in C^\infty_0(W \times (t_1,t_2))$ with $0 \le \vz \le \vd/2$. 
Then \eqref{tilde-u>f+vd/2} asserts that 
\begin{align*}
\psi(x,t):= \tilde{u}_n(x,t) - \vz(x,t) \in K \quad \text{for each} \quad t \in [0,T]. 
\end{align*}
Taking this $\psi$ as $\vp$ in \eqref{vari-ineq-1} and integrating it with respect to $t$ on $(0,T)$, we obtain 
\begin{align}
\int^T_0 \int_\Omega \Delta u_{i,n}(x) \vz(x,t) \, dx dt 
 \le -\int^T_0 \int_\Omega V_{i,n}(x) \vz(x,t)\, dx dt. 
\end{align}
From the definition \eqref{def-mu-n}, the inequality can be reduced to 
\begin{align} \label{pre-mu-n=0}
\sum^{n}_{i=1} \int^{i \tau_n}_{(i-1) \tau_n} \int_\Omega \vz(x,t) d \mu_n dt \le 0. 
\end{align} 
Since $\mu_n \ge 0$, we see that the integral in \eqref{pre-mu-n=0} must be equal to $0$, i.e., 
\begin{align} \label{mu_n=0}
\mu_n(W\times (t_1,t_2)) =0 \,. 
\end{align}
It follows from \eqref{mu_n-ubd} that 
\begin{align*}
\| \mu_n \|_{\mM(\Omega \times (0,T))} := \int^T_0 \!\!\! \int_\Omega d \mu_n dt < C. 
\end{align*}
Thus we deduce that $\mu_n$ converges to $\mu_t$ weakly in $\mM(\Omega \times (0,T))$, i.e., 
\begin{align*}
\int^T_0 \!\!\! \int_\Omega \vp(x,t) d\mu_n dt 
  \to \int^T_0 \!\!\! \int_\Omega \vp(x,t) \, d\mu dt
\end{align*}
for any $\vp \in C^\infty_0(\Omega \times (0,T))$. This fact also yields that 
\begin{align} \label{lower-conti-mu}
\| \mu \|_{\mM(\Omega \times (0,T))} \le \liminf_{n \to +\infty} \| \mu_n \|_{\mM(\Omega \times (0,T))}. 
\end{align}
Combining \eqref{mu_n=0} with \eqref{lower-conti-mu}, we conclude that 
\begin{align}
\mu(W\times (t_1,t_2))=0\,, 
\end{align}
which completes the proof. 
\qed



\end{document}